\nonstopmode
\documentclass[10pt]{amsart}
\usepackage{graphicx}
\usepackage{latexsym}
\usepackage{fancyhdr}
\usepackage{amsmath, amssymb, amsthm}
\usepackage[all]{xy}
\usepackage{pdflscape}
\usepackage{longtable}
\usepackage{rotating}
\usepackage{verbatim}
\usepackage{hyperref}
\hypersetup{
    linkcolor = blue,
    citecolor = blue,
    urlcolor  = blue,
    colorlinks = true,
}

\usepackage{cleveref}
\usepackage{subfigure}
\usepackage{mathrsfs}
\usepackage{mdwlist}
\usepackage{dsfont}
\usepackage{mathtools}
\usepackage{float}
\usepackage{color}
\usepackage{tensor}
\usepackage{stmaryrd}

\usepackage{tkz-base}
\usepackage{pgfplots}
\pgfplotsset{compat=newest}

\usepackage{tkz-euclide}
\usepackage{etoolbox}

\definecolor{teal}{rgb}{0.0, 0.5, 0.5}

\newcounter{mnotecount}[section]

\newcommand{\rmnote}[1]{}

\allowdisplaybreaks

\DeclareFontFamily{U}{mathb}{\hyphenchar\font45}
\DeclareFontShape{U}{mathb}{m}{n}{
      <5> <6> <7> <8> <9> <10> gen * mathb
      <10.95> mathb10 <12> <14.4> <17.28> <20.74> <24.88> mathb12
      }{}
\DeclareSymbolFont{mathb}{U}{mathb}{m}{n}
\DeclareFontSubstitution{U}{mathb}{m}{n}

\let\dot\relax
\DeclareMathAccent{\dot}{0}{mathb}{"39}
\let\ddot\relax
\DeclareMathAccent{\ddot}{0}{mathb}{"3A}
\let\dddot\relax
\DeclareMathAccent{\dddot}{0}{mathb}{"3B}
\let\ddddot\relax
\DeclareMathAccent{\ddddot}{0}{mathb}{"3C}

\theoremstyle{plain}
\newtheorem*{theorem*}{Theorem}
\newtheorem{theorem}{Theorem}[section]
\newtheorem*{lemma*}{Lemma}
\newtheorem{lemma}[theorem]{Lemma}
\newtheorem*{assumption*}{Assumption}

\newtheorem*{proposition*}{Proposition}
\newtheorem{proposition}[theorem]{Proposition}
\newtheorem*{corollary*}{Corollary}
\newtheorem{corollary}[theorem]{Corollary}
\newtheorem*{claim*}{Claim}

\newtheorem*{conjecture*}{Conjecture}

\newtheorem*{question*}{Question}
\newtheorem*{result*}{Result}

\theoremstyle{definition}
\newtheorem*{definition*}{Definition}
\newtheorem{definition}[theorem]{Definition}
\newtheorem*{example*}{Example}

\newtheorem*{algorithm*}{Algorithm}
\newtheorem*{remark*}{Remark}
\newtheorem*{remarks*}{Remarks}
\newtheorem{remark}[theorem]{Remark}

\newtheorem*{convention*}{Convention}

\theoremstyle{plain}

\def\al{\alpha}
\def\be{\beta}
\def\ga{\gamma}
\def\de{\delta}
\def\ep{\epsilon}

\def\ze{\zeta}

\def\th{\theta}

\def\la{\lambda}

\def\rh{\rho}

\def\ta{\tau}

\def\vh{\varphi}

\def\Ga{\Gamma}

\def\Si{\Sigma}

\def\C{\mathbb{C}}

\def\N{\mathbb{N}}

\def\R{\mathbb{R}}
\def\Z{\mathbb{Z}}

\def\cC{\mathcal{C}}

\def\cF{\mathcal{F}}

\def\cH{\mathcal{H}}

\def\cL{\mathcal{L}}

\def\cR{\mathcal{R}}

\def\fd{\mathfrak{d}}

\def\fn{\mathfrak{n}}

\def\p{\partial}

\def\<{\langle}
\def\>{\rangle}
\renewcommand{\o}{\circ}

\def\CN{C_\nn}
\def\nn{\mathbf{N}}

\def\ol{\overline}

\let\on=\operatorname

\newcommand{\sr}[1]%
{\ifmmode{}^\dagger\else${}^\dagger$\fi\ifvmode
\vbox to 0pt{\vss
 \hbox to 0pt{\hskip\hsize\hskip1em
 \vbox{\hsize3cm\raggedright\pretolerance10000
 \noindent #1\hfill}\hss}\vss}\else
 \vadjust{\vbox to0pt{\vss%
 \hbox to 0pt{\hskip\hsize\hskip1em%
 \vbox{\hsize3cm\raggedright\pretolerance10000%
 \noindent #1\hfill}\hss}\vss}}\fi%
}

\providecommand{\mapsfrom}{\kern.2em%
\setbox0=\hbox{$\leftarrow$\kern-.10em\rule[0.26mm]{0.1mm}{1.3mm}}\box0%
\kern.3em}

\DeclareMathOperator*{\esssup}{ess\,sup}
\DeclareMathOperator*{\essinf}{ess\,inf}

\title[Tame properties of differentiable functions]
{Quantitative tame properties of differentiable functions with controlled derivatives}

\author[A.~Rainer]{Armin Rainer}

\address{Fakult\"at f\"ur Mathematik, Universit\"at Wien,
Oskar-Morgenstern-Platz~1, A-1090 Wien, Austria}
\email{armin.rainer@univie.ac.at}

\begin{document}

\begin{abstract}
We show that differentiable functions, defined on a convex body $K \subseteq \R^d$,
whose derivatives are controlled by 
a suitable given sequence of positive real numbers share many properties with polynomials.
The role of the degree of a polynomial is hereby played by an integer associated
with the given sequence of reals, the diameter of $K$, and a real parameter linked to the $\cC^0$-norm
of the function. We give quantitative information on the size of the zero set,
show that it admits a local parameterization by Sobolev functions,
and prove an inequality of Remez-type. From the latter, we deduce several consequences,
for instance, a bound on the volume of sublevel sets and a comparison of $L^p$-norms
reversing H\"older's inequality. The validity of many of the results
only depends on the derivatives up to some finite order; the order can be specified
in terms of the given data.
\end{abstract}

\thanks{Supported by FWF-Project P 32905-N}
\date{\today}

\maketitle


\section{Introduction}

Yomdin showed in \cite{Yomdin:1984dg} that the zero set $Z_f$ of a non-zero
$\cC^\infty$-function $f : B \to \R$ defined on the
unit ball $B \subseteq \R^d$ behaves in many respects as the zero set of a polynomial,
if the partial derivatives of some order $j \ge 2$ are sufficiently small on $B$.
To be precise, let $\|f\|_{j,B} := \sum_{|\al|=j} \frac{j!}{\al!}\|f^{(\al)}\|_B$
and $\|f\|_B := \|f\|_{0,B} = \sup_{x \in B} |f(x)|$.
If, for some $j \ge 2$,
\begin{equation} \label{eq:Yomdin}
   \|f\|_{j,B} \le \frac{1}{2^{j+1}} \|f\|_{0,B},
\end{equation}
then $Z_f$ has the following properties:
\begin{enumerate}
  \item[(a)] There is some ball $B' \subseteq B$ such that for each affine line $\ell$ in $\R^d$
  that meets $B'$ the restriction $f|_\ell$ has at most $j-1$ zeros (counted with multiplicities).
  \item[(b)] $Z_f$ is contained in a countable union of compact $\cC^\infty$-hypersurfaces.
  \item[(c)] The $(d-1)$-dimensional Hausdorff measure of $Z_f$ satisfies
  \begin{equation*}
     \cH^{d-1}(Z_f) \le C(d,j),
  \end{equation*}
  where $C(d,j)$ is a positive constant depending only on $d$ and $j$.
\end{enumerate}
An important fact (which is needed to get (c)) is that the radius of $B'$ depends only on $j$.

This result is in drastic contrast to the classical result that any closed subset of $\R^d$
is the zero set of some $\cC^\infty$-function.
In fact, by (a), a non-zero function satisfying \eqref{eq:Yomdin} cannot have points of infinite flatness.

In this paper, we will show that the zero sets of smooth functions with a different type of
constraints have similar quantitative tame properties.
Instead of assuming that the Fr\'echet derivative of \emph{some} order is ``small'', as in \eqref{eq:Yomdin},
we ask that the growth of the sequence of derivatives of \emph{all} orders is controlled.

In a spirit similar to \cite{Yomdin:1984dg},
Yomdin proved in \cite{Yomdin:2014aa} a Remez-type inequality for smooth functions involving a ``remainder term''
expressible through bounds on the derivatives.
We will prove a Remez-type inequality (without ``remainder term'') and
deduce several consequences for smooth functions with controlled
derivatives.

\subsection{Functions with controlled derivatives}

Let a positive increasing sequence $(\mu_j)_{j \ge 1}$ and a positive real number $M_0>0$ be given and set
$M_j:= M_0 \mu_1 \mu_2 \cdots \mu_j$ for all $j\ge 1$.
Let $K \subseteq \R^d$ be a $d$-dimensional convex body (i.e., compact with non-empty interior $K^\o$).
We will study $\cC^\infty$-functions $f : K \to \R$ such that
\begin{equation} \label{eq:DC}
   \|f\|_{j,K} \le M_j, \quad j \in \N.
\end{equation}
Now it is well-known that the qualitative behavior of functions satisfying \eqref{eq:DC}
--- let us call them \emph{$(M_j)_j$-smooth functions} ---
fundamentally depends on the convergence or divergence of the series $\sum_{j} \frac{1}{\mu_j}$ (the
\emph{quasianalyticity threshold}).
The series $\sum_{j} \frac{1}{\mu_j}$ and its partial sums play a central role in our
quantitative analysis. In fact, this analysis is based on a quantity $\fd_{(\de_K \mu_j)_j}(b)$
which depends on the sequence $(\de_K \mu_j)_j$ and an additional parameter $b>0$.
Here $\de_K:= \on{diam}(K)$ is the diameter of $K$.
Roughly speaking, $\fd_{(\de_K \mu_j)_j}(b)$ is the greatest integer $n$ (possibly infinite) such that
\[
\sum_{j = j_0(b)+1}^n \frac{1}{\mu_j}< \de_K e,
\]
where the integer $j_0(b)$ is explicitly computed from $b$; see \Cref{sec:terminology} for precise definitions.
We call $\fd_{(\de_K\mu_j)_j}(b)$ the \emph{$(\de_K \mu_j)_j$-degree}, since it has similar properties as the degree of
a polynomial. Note that it is closely related to the \emph{Bang degree} introduced by
Nazarov, Sodin, and Volberg \cite{NazarovSodinVolberg04}.

The $(\de_K \mu_j)_j$-degree $\fd_{(\de_K \mu_j)_j}(b)$ is finite if $\sum_{j \ge j_0(b)+1} \frac{1}{\mu_j}$ exceeds $\de_K e$
which is always the case provided that the series $\sum_{j} \frac{1}{\mu_j}$ diverges.
But also if $\sum_{j} \frac{1}{\mu_j}$ converges to a large enough sum,
$\fd_{(\de_K \mu_j)_j}(b)$ is finite and contains useful information on functions satisfying \eqref{eq:DC}.
In fact, this may occur if $\de_K$ is relatively small.

\subsection{The zero set of functions with controlled derivatives}

We will see that the zero set $Z_f$ of $(M_j)_j$-smooth functions $f : K \to \R$
has properties similar to (a), (b), and (c) whenever the $(\de_K \mu_j)_j$-degree
$\fd_{(\de_K \mu_j)_j}(\tfrac{b}{2M_0})$ is finite.
The parameter $b>0$ hereby acts as a lower bound on the $C^0$-norm of $f$, that is, we require that $\|f\|_K \ge b$.

Let us summarize our results in the following statements (A)--(D):
\begin{enumerate}
  \item[(A)] There is a ball $B$ contained in the interior $K^\o$ of $K$ 
      whose radius depends only on $K$ and
  the ratio $\frac{b}{M_1}$
  such that for each affine line $\ell$ in $\R^d$
  that meets $B$ the restriction $f|_\ell$ has at most $2 \fd_{(\de_K \mu_j)_j}(\tfrac{b}{2M_0})$ zeros (counted with multiplicities).
  (\Cref{prop:numberofzeros})
\end{enumerate}
The key to this result is a bound for the number of zeros of $(M_j)_j$-smooth univariate functions which goes back to
Bang \cite{Bang53}; see also \cite{NazarovSodinVolberg04} for an
exposition of Bang's ideas. We revisit Bang's result in some detail in \Cref{sec:univariate}.

As a consequence of Malgrange's preparation theorem and Yomdin's observation \cite[Lemma 6]{Yomdin:1984dg}
we obtain:
\begin{enumerate}
  \item[(B)] $Z_f$ is contained in a countable union of compact $\cC^\infty$-hypersurfaces. (\Cref{prop:Hausdorff})
\end{enumerate}
Then, using a Crofton-type argument, (A) and (B) allow to conclude:
\begin{enumerate}
    \item[(C)] The $(d-1)$-dimensional Hausdorff measure of $Z_f$ satisfies
    \begin{equation*}
       \cH^{d-1}(Z_f) \le C\, \fd_{(\de_K \mu_j)_j}(\tfrac{b}{2M_0}) \de_K^{d-1},
    \end{equation*}
    where the constant $C>0$ depends only on $d$ and the ratio $\frac{\de_B}{\de_K}$;
    $B$ is the ball from (A). (\Cref{prop:Hausdorff})
\end{enumerate}
Combining Malgrange's preparation theorem with results of Parusi\'nski and Rainer \cite{ParusinskiRainer15},
we find that $Z_f$ locally admits a Sobolev parameterization:
\begin{enumerate}
  \item[(D)] There is a finite cover of $K$ by rectangular boxes $U$ with corresponding orthogonal
  coordinates $(x_1,x_2,\ldots,x_d) = (x',x_d)$
  such that $Z_f \cap U$ is contained in the graphs $\{x_d = \xi_i(x')\}$
  of at most $2 \fd_{(\de_K \mu_j)_j}(\tfrac{b}{2M_0})$ continuous functions $\xi_i$
  of Sobolev class $W^{1,p}$, for all $1 \le p < p_0$, where
  \[
    p_0 := \frac{2 \fd_{(\de_K\mu_j)_j}(\tfrac{b}{2M_0})}{2 \fd_{(\de_K\mu_j)_j}(\tfrac{b}{2M_0})-1}.
  \]
  (\Cref{prop:Sobolev})
\end{enumerate}
That means that the partial derivatives $\p_k \xi_i$ of first order of $\xi_i$ exist almost everywhere,
agree with the weak partial derivatives, and $\xi_i$ as well as $\p_k \xi_i$ are in the Lebesgue space $L^p$.

As a by-product, our reasoning shows that (D) remains true for $\cC^\infty$-functions satisfying \eqref{eq:Yomdin}
for some $j\ge 2$ with the role of $2 \fd_{(\de_K \mu_j)_j}(\tfrac{b}{2M_0})$
replaced by $j-1$.

We discuss uniformity of these results at the end of \Cref{sec:multidim}.

\subsection{A Remez-type inequality for functions with controlled derivatives}

The classical Remez inequality \cite{Remes:1936wl} states that for a Lebesgue measurable subset $E$ of $[0,1]$ with
Lebesgue measure $|E|>0$
we have
\[
  \|p\|_{[0,1]} \le T_n\Big(\frac{2-|E|}{|E|}\Big)\, \|p\|_E
\]
for all polynomials $p \in \R[x]$ with degree $\deg p \le n$.
Here $T_n(x)= \cos (n \arccos x)$ is the $n$-th Chebyshev polynomial.
In higher dimensions, there is the generalization due to Yu.\ Brudnyi and Ganzburg \cite{Brudnyi:1973un}:
for a convex body $K \subseteq \R^d$ and a Lebesgue measurable subset $E\subseteq K$ with $|E|>0$
we have the sharp inequality
\[
  \|p\|_{K} \le T_n\Big(\frac{1+(1-\frac{|E|}{|K|})^{1/d}}{1-(1-\frac{|E|}{|K|})^{1/d}}\Big)\, \|p\|_E
\]
for all polynomials $p \in \R[x_1,\ldots,x_d]$ with degree $\deg p \le n$. It implies
\[
  \|p\|_{K} \le \Big(\frac{4d\, |K|}{|E|}\Big)^n \, \|p\|_E.
\]
A version of the latter inequality for analytic functions was obtained by A.\ Brudnyi \cite{Brudnyi_1999}, where the role of $n$
is played by the \emph{analytic degree}; cf.\ \Cref{sec:analyticdegree}.
As already mentioned there is a version for smooth function due to \cite{Yomdin:2014aa}
which however contains a ``remainder term''.

In this paper, we are interested in Remez-type inequalities for $(M_j)_j$-smooth functions $f : K \to \R$
(without ``remainder term'').
On the unit interval $K=[0,1]$, such an inequality was obtained in \cite[Theorem B]{NazarovSodinVolberg04}.
We recall and slightly adjust this result in \Cref{thm:NSV}.
The crucial difference is that, instead of the Bang degree, we work with the $(\de_K \mu_j)_j$-degree
$\fd_{(\de_K \mu_j)_j}(\tfrac{b}{M_0})$
which in a sense allows more ``leeway'': it works for all  $(M_j)_j$-smooth functions $f : K \to \R$ with $\|f\|_K \ge b$.
(In contrast to above, there is no factor $2$ in the denominator.)

In this way, we are able to apply the strategy of restriction to $1$-dimensional sections of Brudnyi and Ganzburg
and obtain a Remez-type inequality for $(M_j)_j$-smooth functions in several variables (\Cref{thm:Remez}).
Actually, the Remez-type inequality is invariant under the $\R^*$-action on $f$ by multiplication with non-zero constants.
This observation leads to a useful version (\Cref{cor:Remez2}) 
for $(M_j)_j$-smooth functions $f$, where $M_j = \|f\|_K \cdot \mu_1 \cdots \mu_j$, for $j \ge 1$, and $M_0=\|f\|_K>0$,
and with $(\de_K \mu_j)_j$-degree $\fd_{(\de_K \mu_j)_j}(1)$.

From this version, we deduce (in a standard way, see for instance
\cite{Brudnyi:1999tb,Brudnyi_1999} and \cite{Brudnyi:1973un})
a number of consequences for functions with controlled derivatives:
\begin{enumerate}
  \item A bound for the volume of sublevel sets. (\Cref{cor:Remez3})
  \item A comparison of $L^p$-norms (reversing H\"older's inequality). (\Cref{cor:Remez4})
  \item A bound for the mean oscillation of $\log |f|$. (\Cref{cor:Remez6} and \Cref{cor:BMO})
\end{enumerate}

Generally, we deduce the main multivariate results from respective univariate ones by 
restriction to affine lines. See \Cref{rem:sectioning} for how this relates to 
results of Bochnak--Siciak type for controlled functions.

\subsection{Complementary results}

In \Cref{sec:polynomialdegree} and \Cref{sec:analyticdegree}, we attempt to clarify the relation between
the $(\mu_j)_j$-degree on the one hand and polynomial and analytic degree on the other hand.
The $(\mu_j)_j$-degree is a more general object, but in a polynomial or analytic setting it is generally
larger than the polynomial or analytic degree, respectively. This is related to the general assumption
that $(\mu_j)_j$ is increasing, which is crucial for many of the results.

As already mentioned,
the $(\mu_j)_j$-degree can be finite also in the \emph{non-quasianalytic} setting (i.e., when $\sum_j \frac{1}{\mu_j}$ converges).
In fact, we obtain in \Cref{cor:cpsupp} a quantitative necessary condition for the existence of
$\cC^\infty$-functions with uniform bounds,
like \eqref{eq:DC}, and compactly supported in
a given ball, in terms of the size of the ball and the $C^0$-norm of the function.
For completeness, we sketch a proof of the fact that, if $\sum_j \frac{1}{\mu_j}< \infty$,
each non-empty closed subset of $\R^d$ is the zero set of a
$\cC^\infty$-function $f$ such that $\|f\|_{j,\R^d} \le A^{j+1} M_j$ for all $j \in \N$ and some $A>0$ (\Cref{rem:complement}).

Bang's univariate result, i.e.\ \Cref{thm:Bang}, not only contains useful information on the numbers of zeros
but also on the locus of points, where derivatives of higher order vanish.
For instance, we get a uniform bound for the number of critical points on certain affine lines (\Cref{prop:critical}).

\subsection{Finite determinacy}

So far we only considered $\cC^\infty$-functions with controlled derivatives of all orders.
Actually, many of our results depend only on the derivatives up to some finite order
and remain true even for functions that are only differentiable up to said finite order.
We will make this precise below, in particular, we will specify the finite order.
To account for this fact, we introduce a customized terminology in \Cref{sec:terminology},
\Cref{sec:finitedeterminacy}, and \Cref{sec:homogeneous}.

\subsection{Notation}

We often write $|E|$ for the Lebesgue measure $\cL^d(E)$ of a measurable set $E \subseteq \R^d$.
If not stated otherwise, ``measurable'' always means ``Lebesgue measurable''.
The $k$-dimensional Hausdorff measure of $E$ is denoted by $\cH^k(E)$.

Let $B_r(a) := \{x \in \R^d : |x-a|<r\}$
and $\ol B_r(a) := \{x \in \R^d : |x-a|\le r\}$ denote the open and closed Euclidean ball
in $\R^d$ centered at $a$ with radius $r>0$, respectively.
If $a$ is the origin, we simply write
$B_r := B_r(0)$ and $\ol B_r := \ol B_r(0)$.

Throughout the paper, we write $\de_K := \on{diam}(K)$ for the diameter of a set $K \subseteq \R^d$.
We denote by $\ol K$, $K^\o$, and $\p K$ the closure, the interior, and the boundary of $K$ in $\R^d$, respectively.

The integral part of a real number $x$ is denoted by $\lfloor x\rfloor := \max\{n \in \Z : n \le x\}$.
Similarly, $\lceil x\rceil := \min\{n \in \Z : n \ge x\}$, but most of the time we will use
$\lceil x\rceil_{\N} := \min\{n \in \N : n \ge x\}$, where $\N := \Z_{\ge 0} := \{n \in \Z : n\ge 0\}$.
Similarly, we use
$\N_{\ge m} := \{n \in \N : n\ge m\}$,
$\R_{>a} := \{x \in \R : x>a\}$, and variations thereof.

\section{$M^{|N}$-smooth functions and $\mu^{|N}$-degree} \label{sec:terminology}

\subsection{Admissible weights}

Let $(\mu_j)_{j \ge 1}$ be an infinite increasing sequence of elements in $\R_{>0} \cup \{\infty\}$, i.e.,
$0< \mu_1 \le \mu_2 \le \cdots$. We allow that sequence elements attain the value $\infty$.
By the requirement that the sequence is increasing,
if $\mu_{j_0} = \infty$, then $\mu_{j} = \infty$ for all $j \ge j_0$.
We say that such a sequence $(\mu_k)_{k \ge 1}$ is an \emph{admissible weight}.
It will be convenient to keep track of the index (if any), where the sequence ceases to be finite. So we write
\begin{align*}
  \mu^{|N} &:= (\mu_j)_{j\ge 1} = (\mu_1,\ldots,\mu_N,\infty,\infty,\ldots) \quad \text{ if } \mu_N <\infty
  = \mu_{N+1} \text{ for } N \ge 1,
  \\
  \mu^{|0} &:= (\mu_j)_{j\ge 1} = (\infty,\infty,\ldots) \quad \text{ if } \mu_j =\infty \text{ for all } j \ge 1,
  \\
  \mu^{|\infty} &:= (\mu_j)_{j\ge 1}  \quad \text{ if } \mu_j <\infty \text{ for all } j \ge 1.
\end{align*}
From now on,
we let $N$ be an element of $\N_{\ge 1} \cup \{\infty\}$ so that the notation $\mu^{|N}$ also includes the case $\mu^{|\infty}$.
(We exclude the case $N=0$, where the weight is worthless.)
Let us adopt the usual conventions for the arithmetic with $\infty$:
$\infty \pm 1 = \infty$, $\frac{1}{\infty}=0$, and $r \cdot \infty = \infty$ if $r\in \R_{>0} \cup \{\infty\}$.

If $\mu^{|N}$ is an admissible weight and $M_0>0$ is any positive real number, then
we call the pair $(\mu^{|N},M_0)$ a \emph{full admissible weight}.
With the full admissible weight $(\mu^{|N},M_0)$ we associate a sequence $M^{|N} = (M_j)_{j\ge 0}$ by setting
\[
  M_j := M_0 \mu_1 \mu_2 \cdots \mu_j, \quad  j \ge 1.
\]
Then $M_j = \infty$ if $N < \infty$ and $j > N$.
That $\mu^{|N}$ is increasing amounts to the property $M^2_j \le M_{j-1} M_{j+1}$ for $j \ge 1$.
Given a sequence $M^{|N} = (M_j)_{j\ge 0}$ with this property, we may recover the pair $(\mu^{|N},M_0)$ in a unique way by
\[
  \mu_j := \frac{M_j}{M_{j-1}} \quad \text{ for } 1\le j \le N \quad \text{ and } \quad \mu_j := \infty \quad \text{ for } j >N,
\]
if $N <\infty$, and simply by $\mu_j := \frac{M_j}{M_{j-1}}$ for all $j \ge 1$ otherwise.
Thus there is a one-to-one correspondence between $(\mu^{|N},M_0)$ and $M^{|N}$.
We also call $M^{|N}$ a \emph{full admissible weight}.

Notice that, given a full admissible weight $(\mu^{|N},M_0)$ and positive constants $r,C>0$,
the pair $(r\mu^{|N},CM_0) = ((r\mu_j)_{j \ge 1},CM_0)$ is again a full admissible weight;
it corresponds to $(C r^j M_j)_{j\ge 0}$.

\subsection{$M^{|N}$-smooth functions} \label{ssec:Msmooth}

Let $M^{|N}$ be a full admissible weight.
Let $K \subseteq \R^d$ be a convex body and $f : K \to \R$ a function.
We say that $f$ is \emph{$M^{|N}$-smooth} (or \emph{$(\mu^{|N},M_0)$-smooth}) if
$f$ is of class $\cC^{N}$ (in an open neighborhood of $K$) and
\begin{equation} \label{eq:Bd1}
   \|f\|_{j,K} \le M_j, \quad 0 \le j < N+1,
\end{equation}
where we set
\[
\|f\|_{j,K} := \sum_{|\al|=j} \frac{j!}{\al!} \|f^{(\al)}\|_K
\]
and $\|f\|_K = \|f\|_{0,K} = \sup_{x \in K} |f(x)|$ denotes the sup-norm.
If $N=\infty$, this means that $f$ is of class $\cC^\infty$ and $\|f\|_{j,K} \le M_j$ for all $j \in \N$.

\subsection{The $\mu^{|N}$-degree} \label{sec:degree}

Let $\mu^{|N}$ be an admissible weight.
We consider the function $\Si_{\mu^{|N}} : \N_{\ge 1} \times \N_{\ge 1}  \to [0,\infty)$
defined by
\[
  \Si_{\mu^{|N}}(m,n) := \sum_{j =m}^n \frac{1}{\mu_j}
\]
and extend the definition to $(\N_{\ge 1} \cup \{\infty\}) \times (\N \cup \{\infty\})$ by setting
\begin{align*}
      \Si_{\mu^{|N}}(m,n) :=
      \begin{cases}
        0 & \text{ if } 1 \le m \le \infty,~ n = 0,
        \\
        \sum_{j =m}^\infty \frac{1}{\mu_j} & \text{ if } 1 \le m < \infty,~ n = \infty,
        \\
        0 &  \text{ if } m = \infty, ~ 1 \le n \le \infty.
      \end{cases}
\end{align*}
Note that, if $m> n$, then $\Si_{\mu^{|N}}(m,n)$ is an empty sum and hence has the value $0$.
For $n \ge N$, we have $\Si_{\mu^{|N}}(m,n) = \Si_{\mu^{|N}}(m,N)$ (since $\frac{1}{\infty}=0$).
The map $(m,n)\mapsto \Si_{\mu^{|N}}(m,n)$ is increasing in $n$
and decreasing in $m$ (even strictly in the range, where $\mu_j$ is finite).

For $b>0$ let $j_0(b)$ be the smallest integer $j\in \N$ with $j \ge \log b^{-1}$, i.e.,
\[
  j_0(b):= \lceil \log b^{-1}\rceil_{\N}.
\]

We define the \emph{$\mu^{|N}$-degree}
$\fd_{\mu^{|N}}(b) \in \N \cup \{\infty\}$
by setting
\begin{align*}
  \fd_{\mu^{|N}}(b) := \sup\big\{n \in \N:  \Si_{\mu^{|N}}(j_0(b)+1,n)
  < e \big\}.
\end{align*}
We remark that $\fd_{\mu^{|N}}(b) \ge j_0(b)$
and $\fd_{\mu^{|N}}(b) = j_0(b)$ occurs precisely if $\mu_{j_0(b)+1}\le  1/e$.
Clearly, $\fd_{\mu^{|N}}(b) = \infty$ if $\Si_{\mu^{|N}}(j_0(b)+1,n) < e$ for all $n$.
The map $\R_{>0} \times \R_{>0} \ni (a,b) \mapsto \fd_{a\mu^{|N}}(b)$ is increasing in $a$ and decreasing in $b$.

  We could assign a $\mu^{|N}$-degree to a $(\mu^{|N},M_0)$-smooth function $f : [0,1] \to \R$ by
  putting
  \[
     \mu^{|N}\text{-degree of } f  := \fd_{\mu^{|N}}(\tfrac{\|f\|_{[0,1]}}{M_0});
  \]
  like the \emph{Bang degree} in \cite{NazarovSodinVolberg04}.
  But the unspecified argument $b$ allows for more flexibility;
  it will however always be related to the ratio of the $\cC^0$-norm of $f$ by $M_0$.

If the domain $K$ of $f$ is a $d$-dimensional convex body,
then, in our results, the diameter $\de_K$ of the domain $K$ (not the dimension!)
will affect the $\mu^{|N}$-degree as a multiplicative factor of the sequence $\mu^{|N}$, i.e.,
\[
  \fd_{\de_K \mu^{|N}}(b).
\]
See, for instance, the proof of \Cref{prop:numberofzeros},
where the function defined on $K$ is restricted to affine lines and 
a rescaling allows to consider it on the interval $[0,1]$.

\section{Zeros of univariate functions} \label{sec:univariate}

In this section, we revisit some results of Bang \cite{Bang53} (see also \cite{NazarovSodinVolberg04}).
Since we have to rephrase them in our terminology, we give detailed proofs.

\subsection{Bang's metric theory revisited}

Let $M^{|N}= (M_j)_{j\ge 0}$ be a full admissible weight.
Let $I \subseteq \R$ be a non-trivial compact interval and
$f : I \to \R$ an $M^{|N}$-smooth function.

Following Bang \cite{Bang53}, we associate with $f$ and $M^{|N}$ a sequence $(b_n)_{n\ge 0}$ of functions
defined by
\begin{equation} \label{eq:Bn}
   b_{n}(t) = b_{f,M^{|N},n}(t) :=  \sup_{j \ge n} \frac{|f^{(j)}(t)|}{e^j M_j}, \quad t \in I,
\end{equation}
with the interpretation that $\frac{|f^{(j)}(t)|}{e^j M_j} = 0$ if $M_j = \infty$ (even if $f^{(j)}(t)$ is not defined).
In other words, if $N$ is finite, then
$b_{n}(t) = \sup_{n \le j \le N} \frac{|f^{(j)}(t)|}{e^j M_j}$ for $n\le N$ and
$b_n \equiv 0$ for $n >N$.

\begin{lemma} \label{lem:Bang}
The sequence $(b_n)_{n \ge 0}$ has the following properties.
  \begin{enumerate}
    \item[(i)] $e^{-n} \ge b_{n}$ for all $n \ge 0$.
    \item[(ii)]$b_{n-1} \ge b_n$ and $f^{(n-1)}(t_0) = 0$ implies $b_{n-1}(t_0) = b_{n}(t_0)$ for all $n \ge 1$.
    \item[(iii)] For all $k>n$
    and all distinct $t, s \in I$,
    \begin{equation} \label{eq:iii}
       b_{n}(s) < \max \{b_{n}(t) , e^{-k} \}\, e^{e |t-s| \mu_k},
    \end{equation}
    where the right-hand side is interpreted as $\infty$ if $\mu_k=\infty$.
    In particular, all $b_n$ are continuous.
  \end{enumerate}
\end{lemma}

\begin{proof}
  (i) and (ii) are obvious (by \eqref{eq:Bd1}).

  Let us prove (iii).
  Now \eqref{eq:iii} is trivial if $N$ is finite and $k>N$ which means that $\mu_k =\infty$.
  So we may assume that $k< N+1$ and that $\mu_k$ and $M_k$ are finite.
  Let $n \le j < k$ and $t, s \in I$.
  Then, by Taylor's formula, for some $\xi$ between $t$ and $s$,
  \begin{align*}
    \frac{|f^{(j)}(s)|}{e^j M_j}
    &\le \sum_{i = 0}^{k-j-1} \frac{|f^{(j+i)}(t)|\, |t-s|^i}{e^j M_j\, i!}
    + \frac{ |f^{(k)}(\xi)|\, |t-s|^{k-j}}{e^j M_j\, (k-j)!}  \\
    &= \sum_{i = 0}^{k-j-1} \frac{M_{j+i}}{M_j}\frac{|f^{(j+i)}(t)| }{e^{j+i} M_{j+i}} \frac{(e|t-s|)^i}{i!}
    + e^{-k} \frac{M_{k}}{M_j} \frac{ |f^{(k)}(\xi)|}{M_k} \frac{ (e|t-s|)^{k-j}}{(k-j)!} \\
    &\le b_{n}(t) \sum_{i = 0}^{k-j-1} \mu_k^i \frac{(e|t-s|)^i}{i!}
    + e^{-k} \mu_k^{k-j}  \frac{ (e|t-s|)^{k-j}}{(k-j)!} \\
    &< \max\{ b_{n}(t) , e^{-k} \}\, e^{e |t-s| \mu_k},
  \end{align*}
  where we used that $(\mu_j)_j$ is increasing.
  If $j \ge k$, then trivially
  \begin{align*}
    \frac{|f^{(j)}(s)|}{e^j M_j} \le e^{-j} < \max\{ b_{n}(t), e^{-k} \} \, e^{e |t-s| \mu_k}.
  \end{align*}
  This implies \eqref{eq:iii}.

  To see the continuity of the $b_n$, we treat separately the cases $N<\infty$ and $N=\infty$.
  First, if $N$ is finite, then $b_n$, for $n \le N$, is a maximum of finitely many continuous functions and $b_n \equiv 0$, for $n >N$; 
  thus continuity is clear.
  Second, if $N= \infty$, then
  continuity of $b_n$, for all $n \ge 0$, follows from \eqref{eq:iii}:
  fix $t \in I$ and a sequence $t_\nu \to t$ (with $t_\nu \ne t$) in $I$.  
  By \eqref{eq:iii}, for each $k > n$ and all $\nu$, 
  \[
       b_{n}(t_\nu) < \max \{b_{n}(t) , e^{-k} \}\, e^{e |t-t_\nu| \mu_k},
  \]
  whence
  \[
      \limsup_{\nu \to \infty} b_{n}(t_\nu) \le \max \{b_{n}(t) , e^{-k} \}.
  \]
  Since this holds for all $k>n$, we have 
  \[
      \limsup_{\nu \to \infty} b_{n}(t_\nu) \le b_n(t).
  \]
  Again by \eqref{eq:iii}, for each $k > n$ and all $\nu$, 
  \[
       b_{n}(t) < \max \{b_{n}(t_\nu) , e^{-k} \}\, e^{e |t-t_\nu| \mu_k},
  \]
  and so we find 
  \[
       b_{n}(t) \le \max \{\liminf_{\nu \to \infty} b_{n}(t_\nu) , e^{-k} \}
  \]
  for all $k>n$, and therefore
  \[
        b_{n}(t) \le \liminf_{\nu \to \infty} b_{n}(t_\nu). 
  \]
  It follows that $\lim_{\nu \to \infty} b_n(t_\nu) = b_n(t)$.  
\end{proof}

The following proposition is due to \cite{Bang53}; we implement several modifications.

 \begin{proposition} \label{thm:Bang}
   Let $M^{|N}= (M_j)_{j \ge0}$ be a full admissible weight.
   Let $I \subseteq \R$ be a non-trivial compact interval and
   $f : I \to \R$ an $M^{|N}$-smooth function.
   Let $m \in \N$ be such that $m+1\le N$.
   Assume that for all $0\le j \le m$ there is $x_j \in I$ such that $f^{(j)}(x_j) = 0$.
   Let $x_{-1}$ be an arbitrary point in $I$.
   Then,
 	\begin{equation} \label{eq:assertion}
 		\sum_{j=0}^m |x_{j-1} -x_j| \ge \frac{1}{e}\Si_{\mu^{|N}}(j_0+1,m+1),
 	\end{equation}
  where
  \begin{equation} \label{eq:k0}
    j_0 = j_0(b_{f,M^{|N},0}(x_{-1})) = \lceil \log b_{f,M^{|N},0}(x_{-1})^{-1} \rceil_{\N}.
  \end{equation}
  If $j_0\le m$, then the inequality \eqref{eq:assertion} is strict.
 \end{proposition}

 \begin{remark} \label{rem:Bang}
   A few remarks are in order.

   (1) Note that $j_0 <\infty$ if and only if
   there exists $0 \le j \le N$ with $f^{(j)}(x_{-1})\ne 0$.

   (2)
   The right-hand side of \eqref{eq:assertion} is zero
   if $j_0  > m$. In that case, the statement is trivial.
   For instance, if $N=\infty$,
   we do here not exclude the case that $f$ has points of infinite flatness.
   If $x_{-1}$ is such a point, then $j_0 = \infty$ and
   hence
   \eqref{eq:assertion} is trivially true. In that case,
   the assumptions of the proposition are satisfied for the choice $x_j := x_{-1}$, $j = 0,\ldots,m$,
   entailing that also the left-hand side of \eqref{eq:assertion} is zero.
 \end{remark}

 \begin{proof}[Proof of \Cref{thm:Bang}]
   Let $(b_{n})_{n \ge 0}$ be the decreasing sequence of continuous functions \eqref{eq:Bn}; cf.\ \Cref{lem:Bang}.
   We will construct a new continuous function $\be$ by tracing through the graphs of the $b_n$ (for $0 \le n \le m$)
   and switching from $b_n$ to $b_{n+1}$ at $x_n$.
   For $0 \le k \le m$,
   set $\ta_k := \sum_{j=0}^{k} |x_{j-1}-x_{j}|$ and $\ta_{-1}:= 0$.
   For $t \in [\ta_{n-1},\ta_n]$, where $0 \le n \le m$, define
   \[
     \be_{n} (t) :=
     \begin{cases}
        b_{n}(x_{n-1} +\ta_{n-1} -t) & \text{if } x_n < x_{n-1}, \\
        b_{n}(x_{n-1} -\ta_{n-1} +t) & \text{if } x_n \ge x_{n-1}.
     \end{cases}
   \]
   Then each $\be_{n}$ is continuous and
   $\be_{n}(\ta_n) = b_{n}(x_n) = b_{n+1}(x_n) =\be_{n+1}(\ta_n)$; by \Cref{lem:Bang}.
   Thus,
   \[
     \be(t) := \be_{n}(t) \quad \text{ if } t \in [\ta_{n-1},\ta_n],~ 0 \le n \le m,
   \]
   defines a continuous function on $[0,\ta_m]$.
   By \Cref{lem:Bang}, we have $\be(t) \le e^{-n}$ for all $t \ge \ta_{n-1}$
   as well as
   \begin{equation} \label{eq:minB}
     \be(\ta_m) = \be_m(\ta_m) = b_m(x_m) = b_{m+1}(x_m) \le e^{-m-1}.
   \end{equation}
   On the other hand, with $j_0$ as defined in \eqref{eq:k0},
   \begin{equation} \label{eq:forkappa}
     \be(0) = \be_0(\ta_{-1}) = b_{0}(x_{-1}) \ge e^{-j_0}.
   \end{equation}
   It might be that $e^{-j_0} \le e^{-m-1}$ (including the case $j_0=\infty$).
   Then the right-hand side of \eqref{eq:assertion} is zero so that \eqref{eq:assertion} is trivially true.
   Thus we may assume that $j_0 \le m$.
   In view of \eqref{eq:minB} and \eqref{eq:forkappa},
   the range of $\be$ then contains all numbers $e^{-j}$ for $j_0 \le j \le m+1$.
   So we find a strictly increasing sequence $t_j$, for $j_0 \le j \le m+1$,
   such that $\be(t_j) = e^{-j}$
   and $\be(t) > e^{-j}$ if
   $t <t_j$ (starting in the point $(0,\be(0))$ let $t_j$ be the first time that the graph of $\be$ meets the horizontal line
   with ordinate $e^{-j}$).
   Then
   \begin{equation} \label{eq:Bang}
      \be(t_{j-1}) < \be(t_j)\, e^{e (t_j-t_{j-1}) \mu_j}, \quad j_0+1 \le j \le m+1.
   \end{equation}
   To see this, we apply (iii) of \Cref{lem:Bang} to each
   interval in the subdivision of $(t_{j-1},t_j)$ induced by the points $\ta_n$ between $t_{j-1}$ and $t_j$,
   and notice that, since $t_{j} \le \ta_{j-1}$ (as $\be(t) \le e^{-j}$ if $t\ge \ta_{j-1}$), we have $n < j$ for all such $n$ and
   $\max \{b_{n}(t) , e^{-j} \} = b_{n}(t)$ for all $t \in (t_{j-1},t_j)$.

   In view of $\be(t_j) = e^{-j}$, \eqref{eq:Bang}
   amounts to
   \[
      t_j-t_{j-1} > \frac1{e \mu_j}, \quad j_0+1 \le j \le m+1.
   \]
   Summing over $j$, we find
   \[
     t_{m+1}\ge t_{m+1} - t_{j_0} > \frac1{e} \sum_{k= j_0+1}^{m+1} \frac1{\mu_j}.
   \]
   Since $\ta_m  \ge t_{m+1}$, this yields \eqref{eq:assertion}.
   It also shows that the inequality is strict provided that $j_0 \le m$.
 \end{proof}

 \subsection{Bounds for the number of zeros}

 Now it is easy to deduce a lower bound for the length of the interval and
 an upper bound for the number of zeros.

 \begin{corollary} \label{cor:Bang1}
   Let $M^{|N}= (M_j)_{j\ge 0}$ be a full admissible weight.
   Let $I \subseteq \R$ be a non-trivial compact interval and
   $f : I \to \R$ an $M^{|N}$-smooth function.
   Let $z_1 \le z_2 \le \cdots \le z_m$ be an increasing enumeration of some of the zeros of $f$, where $m \le N$,
   and let $x_{-1} \in I \setminus (z_1,z_m)$ be arbitrary.
   Then we have a lower bound for the length $|I|$ of $I$,
   \begin{equation} \label{eq:lowerbound}
     |I| >  \frac{1}{e} \Si_{\mu^{|N}}(j_0+1,m), \qquad (j_0 = \lceil \log b_{f,M^{|N},0}(x_{-1})^{-1} \rceil_{\N}),
   \end{equation}
   and an upper bound for the number of zeros,
   \begin{equation} \label{eq:nMf}
      m \le \fd_{|I|\mu^{|N}}(b_{f,M^{|N},0}(x_{-1})).
   \end{equation}
 \end{corollary}

 \begin{proof}
   Suppose that $x_{-1}\le z_1$; if $x_{-1} \ge z_m$ the proof is similar.
   By Rolle's theorem, there is a sequence of points $x_{-1} \le x_0= z_1 \le x_1 \le \cdots \le x_{m-1}$
   in $I$ such that $f^{(j)}(x_j) = 0$ for all $0 \le j \le m-1$.
 	By \eqref{eq:assertion},
 	\[
 		|I| \ge \sum_{j=0}^{m-1} (x_j - x_{j-1}) \ge \frac{1}{e} \Si_{\mu^{|N}}(j_0+1,m),
 	\]
  with strict inequality if $j_0 \le m-1$.
  If $j_0 \ge m$, then $\Si_{\mu^{|N}}(j_0+1,m) =0$ and
  also in that case the inequality is strict, since $I$ is assumed to be non-trivial.
  So we proved \eqref{eq:lowerbound}.
  By the definition of $\fd_{\mu^{|N}}$, also \eqref{eq:nMf} follows.
 \end{proof}

 \begin{remark} \label{rem:distancetozeros}
   Variations of the argument yield further useful information.
   For instance, if $f(x_{-1}) \ne 0$ and $x_0$ is an $m$-fold zero of $f$, then
   \[
     |x_{-1}-x_0| \ge \frac{1}{e} \Si_{\mu^{|N}}(j_0+1,m) \ge \frac{1}{e} \Si_{\mu^{|N}}(\lceil \log \tfrac{M_0}{|f(x_{-1})|} \rceil_\N+1,m),
   \]
   since in this case $\sum_{j=0}^{m-1} |x_j - x_{j-1}| = |x_{-1}-x_0|$ and
   $b_{f,M^{|N},0}(x_{-1}) \ge \tfrac{|f(x_{-1})|}{M_0}$.
 \end{remark}

 In \Cref{cor:Bang1},
   the list of zeros $z_j$ may not comprise all zeros of $f$.
   If $z_j$ is in the list and $z_j$ is a multiple zero of $f$, we do not even require that all multiplicities of $z_j$
   appear in the list.

   Under an additional assumption, we find that the total number of zeros is finite and get an upper bound for it.

   \begin{corollary} \label{cor:Bang2}
     Let $M^{|\infty}= (M_j)_{j\ge 0}$ be a full admissible weight.
      Let $I \subseteq \R$ be a non-trivial compact interval and
     $f : I \to \R$ an $M^{|\infty}$-smooth function.
     Let $x_{-1} \in I$. If
     \begin{equation} \label{eq:addass}
        \Si_{\mu^{|\infty}}(j_0+1,\infty) > |I| e, \qquad (j_0 = \lceil \log b_{f,M^{|\infty},0}(x_{-1})^{-1} \rceil_{\N}),
     \end{equation}
     then the total number $m$ of zeros of $f$ in $I$ counted with multiplicities is finite.
     In that case, $\fd_{|I|\mu^{|\infty}}(b_{f,M^{|\infty},0}(x_{-1})) < \infty$ and
     \begin{equation} \label{eq:boundforzeros}
        m \le 2 \fd_{|I|\mu^{|\infty}}(b_{f,M^{|\infty},0}(x_{-1})).
     \end{equation}
     For each $x_{-1} \in I$ satisfying \eqref{eq:addass}
     and not lying strictly between the smallest and the largest zero of $f$, we even have
     \begin{equation}
        m \le \fd_{|I|\mu^{|\infty}}(b_{f,M^{|\infty},0}(x_{-1})).
     \end{equation}
     If $\fd_{|I|\mu^{|\infty}}(b_{f,M^{|\infty},0}(x_{-1})) = 0$, then $f$ has no zeros in $I$.
   \end{corollary}

   \begin{proof}
       By \eqref{eq:addass}, it is obvious that $\fd_{|I|\mu^{|\infty}}(b_{f,M^{|\infty},0}(x_{-1}))<\infty$.
       Suppose for contradiction that there are infinitely many zeros of $f$.
       Then we find an increasing infinite sequence of zeros on the right of $x_{-1}$
       or a decreasing infinite  sequence of zeros on the left of $x_{-1}$ (or on both sides).
       But that contradicts \Cref{cor:Bang1}.
       The upper bounds for positive $m$ follow again from \Cref{cor:Bang1};
       if there are $m_l$ zeros left and $m_r$ zeros right of $x_{-1}$, then
       $m = m_l + m_r \le 2 \fd_{|I|\mu^{|\infty}}(b_{f,M^{|\infty},0}(x_{-1}))$.
       In particular, if $f$ has zeros in $I$, then $\fd_{|I|\mu^{|\infty}}(b_{f,M^{|\infty},0}(x_{-1})) \ne 0$.
   \end{proof}

\begin{remark} \label{rem:bootstrap}
  If in \Cref{cor:Bang2} we assume that $f$ is an $M^{|N}$-smooth function with finite $N$,
  then in general $f$ could have more zeros than $N$ so that \Cref{cor:Bang1} is not applicable.

  On the other hand, in the setting of \Cref{cor:Bang2} we
  conclude that the number of zeros of $f$ left and right of $x_{-1}$
  is bounded by
  \begin{equation} \label{eq:bootstrap}
      N := \fd_{|I|\mu^{|\infty}}(b_{f,M^{|\infty},0}(x_{-1})).
  \end{equation}
  So \emph{a posteriori} we obtain that total number of zeros $m$ of $f$ in $I$
  satisfies
  \begin{equation} \label{eq:bootstrap2}
     m \le 2 \fd_{|I|\mu^{|N}}(b_{f,M^{|N},0}(x_{-1})),
  \end{equation}
  where $N$ is given by \eqref{eq:bootstrap} and $\mu^{|N}$, $M^{|N}$
  are obtained from $\mu^{|\infty}$, $M^{|\infty}$ simply by setting
  all elements with index $\ge N+1$ equal to $\infty$.
  In particular,
  the bound on the number of zeros in \eqref{eq:bootstrap2}
  depends only on the derivatives up order $N$ of
  $f$.
\end{remark}

\subsection{Quasianalyticity}

Let $\mu^{|\infty} = (\mu_j)_{j \ge 0}$ be an admissible weight.
We see that, if
\begin{equation} \label{eq:qa}
  \sum_j \frac{1}{\mu_j} = \infty
\end{equation}
and $f$ is not identically zero,
then condition \eqref{eq:addass} is always satisfied (provided that $j_0$ is finite) and \Cref{cor:Bang2} applies.
An admissible weight $\mu^{|\infty}$ satisfying \eqref{eq:qa} is said to be \emph{quasianalytic};
if \eqref{eq:qa} is not fulfilled we say that $\mu^{|\infty}$ is \emph{non-quasianalytic}
(analogously for full admissible weights $M^{|\infty}$).
If we speak of (non-)quasianalytic admissible weights, we always presuppose that $N=\infty$.

In fact, let $I \subseteq \R$ be a non-trivial compact interval and
  let $\cC^{M^{|\infty}}(I)$ denote the set of all functions $f : I \to \R$ such that
  there exists a constant $\rh>0$ such that $\|f^{(j)}\|_I \le \rh^{j+1} M_j$ for all $j \in \N$.
  Then $\cC^{M^{|\infty}}(I)$ is called \emph{quasianalytic} if, for all $f \in \cC^{M^{|\infty}}(I)$ and $x_0 \in I$,
  triviality of the Taylor series $\widehat f_{x_0}$ of $f$ at $x_0$, i.e., $\widehat f_{x_0}=0$,
  implies that $f$ is identically zero.
  By the Denjoy--Carleman theorem (see e.g.\ \cite{Rainer:2021aa}),
  $\cC^{M^{|\infty}}(I)$ is quasianalytic if and only if $M^{|\infty}$ is quasianalytic
  (that is, \eqref{eq:qa} holds).

  Note that \Cref{cor:Bang2} implies one direction of this equivalence.
  For, assume that $f \in \cC^{M^{|\infty}}(I)$, $x_0 \in I$, and $\widehat f_{x_0}=0$.
  Then there exists $\rh>0$ such that $f$ is $(\rh \mu^{|\infty},\rh M_0)$-smooth.
  If $f$ is not identically zero, we find $x_{-1} \in I$ with $f(x_{-1}) \ne 0$ and thus
  $j_0 = \lceil \log b_{f,(\rh^{j+1}M_j)_j,0}(x_{-1})^{-1}\rceil_{\N}$ is finite.
  Since $\widehat f_{x_0}=0$ and thus the number of zeros of $f$ is infinite,
  \Cref{cor:Bang2} implies that \eqref{eq:addass} must be violated.
  Since $j_0$ is finite, we may infer that
  \eqref{eq:qa} is violated, that is, $M^{|\infty}$ is non-quasianalytic.

\section{The zero set of multivariate functions} \label{sec:multidim}

Let $K \subseteq \R^d$ be a convex body.
We will derive quantitative information on the zero set of $M^{|\infty}$-smooth functions $f : K \to \R$
in terms of the $\de_K \mu^{|\infty}$-degree; recall that $\de_K$ is the diameter of $K$.
It turns out that these results actually depend only on a finite number of derivatives.
To account for this we introduce the following bit of notation.

\subsection{Finite determinacy} \label{sec:finitedeterminacy}

Let $M^{|\infty}= (M_j)_{j\ge 0}$ be a full admissible weight, $K \subseteq \R^d$ a convex body,
and $b>0$.
Assume that
\begin{equation} \label{eq:multdimcond}
   \Si_{\mu^{|\infty}}(j_0+1,\infty) > \de_K e, \qquad (j_0 = \lceil \log (\tfrac{b}{2M_0})^{-1}\rceil_{\N}).
\end{equation}
(Note that \eqref{eq:multdimcond} is in any case satisfied if $M^{|\infty}$ is a quasianalytic full admissible weight.
The factor $2$ in the definition of $j_0$ could be replaced by any real number $>1$ without
changing the validity of the results.)
Then
\begin{equation} \label{eq:defN}
   \nn:= \fd_{\de_K \mu^{|\infty}}(\tfrac{b}{2M_0})
\end{equation}
is a nonnegative (finite) integer.
Let
$\mu^{|\nn+1}$ (resp.\ $M^{|\nn+1}$) be the (resp.\ full) admissible weight
obtained from $\mu^{|\infty}$ (resp.\ $M^{|\infty}$) by setting
all elements with index $\ge \nn+2$ equal to $\infty$.
Then $\Si_{\mu^{|\infty}}(j_0+1,n) = \Si_{\mu^{|\nn+1}}(j_0+1,n)$ for $n \le \nn+1$ so that
\begin{equation} \label{eq:fdet}
    \fd_{\de_K \mu^{|\infty}}(\tfrac{b}{2M_0}) = \fd_{\de_K \mu^{|\nn+1}}(\tfrac{b}{2M_0}) = \nn.
\end{equation}
Furthermore,
\begin{equation} \label{eq:fdet2}
    \Si_{\mu^{|\infty}}(j_0+1,\infty) >\Si_{\mu^{|\nn+1}}(j_0+1,\infty) = \Si_{\mu^{|\nn+1}}(j_0+1,\nn+1) \ge \de_K e
\end{equation}
which follows easily from the definitions.

\begin{definition}[Admissible data] \label{def:admissibledata}
  We call the triple $(M^{|\infty},K,b)$, where
  $M^{|\infty}= (M_j)_{j\ge 0}$ is a full admissible weight, $K \subseteq \R^d$ a convex body,
  and $b>0$, \emph{admissible data} if \eqref{eq:multdimcond} holds.
  In that case, $\nn$ defined by \eqref{eq:defN} is called the \emph{integer associated with the data $(M^{|\infty},K,b)$}.
  In this setting, $\mu^{|\nn+1}$ (resp.\ $M^{|\nn+1}$) will always denote the (resp.\ full) admissible weight
  resulting from $\mu^{|\infty}$ (resp.\ $M^{|\infty}$) by
  setting all elements with index $\ge \nn+2$ equal to $\infty$.
\end{definition}

\subsection{Number of zeros on affine lines}

Recall that $K^\o$ denotes the interior of $K$.

\begin{theorem} \label{prop:numberofzeros}
  Let $(M^{|\infty},K,b)$ be admissible data and $\nn$ the associated integer.
  Let $f : K \to \R$ be any $M^{|\nn+1}$-smooth function such that $\|f\|_{K} \ge b$.
  Then there is a ball $B \subseteq K^\o$,
  whose radius only depends on $K$ and the ratio $\frac{b}{M_1}$,
  such that for each affine line $\ell$ in $\R^d$ that meets $B$
  the restriction $f|_\ell$
  has at most $2 \nn$ zeros.
\end{theorem}

\begin{proof}
  Consider $U := \{x \in K : |f(x)|> \frac{b}{2}\}$.
  Since $\|f\|_{K} \ge b$, there is $a \in U$ with  $|f(a)| \ge b$.
  For all $x \in B_s(a) \cap K$ with $s:= \frac{b}{3 M_1}$, we have
  \[
    |f(x) - f(a)| \le  \|f\|_{1,K}\, |x-a| \le   M_1 s = \tfrac{b}{3},
  \]
  so that
  \[
    |f(x)| \ge |f(a)| - |f(x) - f(a)| \ge \tfrac{2b}{3}.
  \]
  It follows that $a \in B_s(a) \cap K \subseteq U$.
  But the intersection $B_s(a) \cap K^\o$ contains an open ball $B$,
  whose radius depends only on $s$ and on the ``thickness'' of $K$ near $a$.

  Let $\ell \subseteq \R^d$ be any affine line that meets the ball $B$.
  Let $x_0$ and $x_1$ be the intersection points of $\ell$ with the boundary of $K$.
  Then $g : [0,1] \to \R$ given by $g(t): = f(x_0 + t (x_1 -x_0))$ defines a $\cC^{\nn+1}$-function
  satisfying
  \begin{equation*}
     \|g^{(j)}\|_{[0,1]} \le |x_1 - x_0|^j  \|f\|_{j,K}
     \le \de_K^j M_j, \quad  0 \le j \le \nn+1.
  \end{equation*}
  Thus, $g : [0,1] \to \R$ is $(\de_K \mu^{|\nn+1},M_0)$-smooth.
  For all $t \in [0,1]$ with $x_0 + t(x_1 -x_0) \in B$, we find
  \[
    b_{g,(\de_K^jM_j)_j,0}(t)= \sup_{j \ge 0} \frac{|g^{(j)}(t)|}{(\de_Ke)^j M_j} \ge  \frac{|g(t)|}{M_0} > \frac{b}{2 M_0},
  \]
  whence $\fd_{\de_K \mu^{|\nn+1}} (b_{g,(\de_K^jM_j)_j,0}(t)) \le \fd_{\de_K \mu^{|\nn+1}} (\tfrac{b}{2 M_0}) = \nn$ (see \eqref{eq:fdet}).
  Fix such a $t$.
  Suppose that $g$ has (at least) $\nn+1$ zeros left or right of $t$.
  Then \Cref{cor:Bang1} implies
  \[
    \nn+1 \le \fd_{\de_K \mu^{|\nn+1}} (b_{g,(\de_K^jM_j)_j,0}(t)) \le \nn,
  \]
  a contradiction.
  Thus the total number of zeros of $g$, and hence of $f|_{\ell}$, is finite and bounded by $2 \nn$.
\end{proof}

\begin{corollary} \label{cor:numberofzeros}
  In the setting of \Cref{prop:numberofzeros},
  if $K = \ol B_r$, then $\de_K = 2r$ and the radius of $B$ is at least $\frac{1}{2}\min\{\frac{b}{3M_1}, r\}$.
\end{corollary}

\begin{proof}
  The intersection $B_s(a) \cap B_r$ contains an open ball $B$ of radius at least $\frac{1}2 \min\{s,r\}$.
\end{proof}

\begin{remark}
  In analogy to \Cref{rem:distancetozeros},
  we can give at each point in $K$ a lower bound for the distance to the closest zero of $f$ of a specific multiplicity;
  often the bound is trivial, but not always:

  Let $M^{|\nn+1}$ be a full admissible weight and $K \subseteq \R^d$ a convex body.
  Let $f : K \to \R$ be an $M^{|\nn+1}$-smooth function and $x_{-1} \in K \setminus Z_f$.
  Then $f$ has no zeros of multiplicity at least $m$, where $m\le \nn+1$,
  in the ball $B_{\ep_m}(x_{-1})$ with center $x_{-1}$ and radius
  \[
    \ep_m:= \frac{1}{e} \Si_{\de_K \mu^{|\nn+1}}(j_0+1,m), \qquad (j_0 = \lceil \log \tfrac{M_0}{|f(x_{-1})|} \rceil_{\N}).
  \]
  This follows from \Cref{rem:distancetozeros} applied to all affine lines through $x_{-1}$.

  In particular, $f$ has no zeros in $B_{\ep_1}(x_{-1})$.
  But $\ep_1 \ne 0$ only if $j_0 =0$ which is equivalent to $|f(x_{-1})| = M_0$; in that case $\ep_1 = \frac{1}{\de_K e \mu_1}$.
  Furthermore, $f$ has no zeros that are also critical points in $B_{\ep_2}(x_{-1})$, where
  \[
  \ep_2 =
  \begin{cases}
    \frac{1}{\de_K e} (\frac{1}{\mu_1}+\frac{1}{\mu_2}), & \text{ if } |f(x_{-1})| = M_0,
    \\
    \frac{1}{\de_K e \mu_2}, & \text{ if } \frac{M_0}{e}\le  |f(x_{-1})| < M_0,
    \\
    0, & \text{ otherwise}.
  \end{cases}
  \]
\end{remark}

For later use, we remark that the zero set of a function $f$ as in \Cref{prop:numberofzeros} has zero Lebesgue measure.
This will be a trivial consequence of \Cref{prop:Hausdorff}, if $f$ is additionally assumed to be of class $\cC^\infty$,
but we will need it without this assumption.

\begin{corollary} \label{cor:zeromeasure}
  In the setting of \Cref{prop:numberofzeros},
  $\cL^d(Z_f) = 0$.
\end{corollary}

\begin{proof}
   First of all, $f$ is continuous so that $Z_f$ is closed, hence measurable.
   The union of all affine lines of a fixed direction meeting $B$ form an open cylinder $U$.
   By compactness, $K$ is covered by finitely many such cylinders $U$.
   Now $Z_f \cap U$ is measurable, whence we can apply Fubini's theorem to its characteristic function.
   By \Cref{prop:numberofzeros}, this gives $\cL^d(Z_f \cap U)=0$.
\end{proof}

\subsection{Hausdorff measure of the zero set}

\begin{theorem} \label{prop:Hausdorff}
  Let $(M^{|\infty},K,b)$ be admissible data and $\nn$ the associated integer.
  Let $f : K \to \R$ be any $M^{|\nn+1}$-smooth $\cC^\infty$-function such that $\|f\|_{K} \ge b$.
  Then the zero set $Z_f$ is contained in a countable union of compact $\cC^\infty$-hypersurfaces
  and its $(d-1)$-dimensional Hausdorff measure satisfies
  \begin{equation} \label{eq:Hausdorff}
    \cH^{d-1}(Z_f) \le C\, \nn \, \de_K^{d-1},
  \end{equation}
  where the constant $C>0$ depends only on $d$ and on the ratio $\tfrac{\de_B}{\de_K}$ and where
  $B$ is the ball from \Cref{prop:numberofzeros}.
  In the case $K = \ol B_r$, the constant $C$ depends only on $d$ and $\min\{\frac{b}{6 r M_1},\frac{1}2\}$;
  it blows up as $b \to 0$ or $r \to \infty$.
\end{theorem}

\begin{proof}
  The first statement follows from Malgrange's preparation theorem
  and \cite[Lemma 6]{Yomdin:1984dg}.
  Thus,
  we may apply the Crofton-type result \cite[Lemma 7]{Yomdin:1984dg}
  to conclude \eqref{eq:Hausdorff} from \Cref{prop:numberofzeros}.

  If $K = \ol B_r$,
  then we may take $\de_B = \min\{\frac{b}{3 M_1},r\}$, by \Cref{cor:numberofzeros},
  so that $\tfrac{\de_B}{\de_K} = \min\{\frac{b}{6 r M_1},\frac{1}2\}$. This also yields the asymptotics of the constant $C$.
\end{proof}

\begin{remark}
  In \Cref{prop:Hausdorff} (and in the next result),
  we assume that $f$ is of class $\cC^\infty$ in order to apply Malgrange's preparation theorem.
  That $f : K \to \R$ is an $M^{|\nn+1}$-smooth $\cC^\infty$-function means that
  $f$ is of class $\cC^\infty$ and satisfies $\|f\|_{j,K} \le M_j$ for all $0 \le j \le \nn+1$.
\end{remark}

\subsection{Sobolev parameterization of the zero set}

Let $(M^{|\infty},K,b)$ be admissible data and $\nn$ the associated integer.
Let $f : K \to \R$ be any $M^{|\nn+1}$-smooth $\cC^\infty$-function such that $\|f\|_{K} \ge b$.

The collection $\cL$ of all affine lines in $\R^d$ that meet the open ball $B$ from \Cref{prop:numberofzeros}
covers $\R^d$, i.e., $\R^d = \bigcup_{\ell \in \cL} \ell$.
Fix $z_0 \in Z_f$ and any line $\ell \in \cL$ with $z_0 \in \ell$.
By \Cref{prop:numberofzeros}, the multiplicity $n$ of $z_0$ as a zero of $f|_\ell$ is at most $2\nn$.
Let $x_1,\ldots,x_d$ be an orthogonal coordinate system such that $\ell$ coincides with the $x_d$-axis.
By Malgrange's preparation theorem,
there is a rectangular neighborhood $U$ of $z_0$
of the form $U := \{x=(x_1,\ldots,x_d) \in \R^d : \al_j < x_j < \be_j \text{ for all } 1 \le j\le d \}$
such that $f(x) = p(x)u(x)$ for $x \in U$, where $u$ does not vanish on $U$ and
\begin{equation} \label{eq:preparation}
  p(x',x_d) = x_d^n + a_1(x') x_d^{n-1} + \cdots + a_d(x')
\end{equation}
is a polynomial in $x_d$ with $\cC^\infty$-coefficients $a_j$ which are defined on
the projection $U'$ of $U$ onto the first $d-1$ coordinates
$x':=(x_1,\ldots,x_{d-1})$.
Thus,
\[
  Z_f \cap U = Z_p \cap U.
\]

Let $[\ze_1(x'),\ldots,\ze_n(x')]$ be the unordered $n$-tuple of complex roots (with multiplicities)
of the polynomial $p(x',x_d)$ in $x_d$
for any $x' \in U'$.
Note that the number of real roots among the $\ze_j(x')$ may change with varying $x'$.
Let $(\xi_1(x'), \ldots, \xi_n(x'))$ be the $n$-tuple consisting of the real parts of the $\ze_j(x')$,
ordered increasingly such that $\xi_1(x') \le \xi_2(x') \le \cdots \le \xi_n(x')$ for all $x'$.
Since all the real roots among the $\ze_j(x')$ clearly are also among the $\xi_i(x')$
(with correct multiplicities),
it follows that
\begin{align*}
   Z_f \cap U
   &= Z_p \cap U
   \\
   &= \{(x',x_d)\in U : \text{ there is } 1 \le j \le n  \text{ with } x_d = \ze_j(x')\}
   \\
   &\subseteq \{(x',x_d)\in U : \text{ there is } 1 \le i \le n  \text{ with } x_d = \xi_i(x')\}
   \subseteq \bigcup_{i=1}^n \Ga(\xi_i),
\end{align*}
where $\Ga(\xi_i):= \{(x',\xi_i(x')) : x' \in U'\}$ denotes the graph of $\xi_i : U' \to \R$.
The increasing order of the $\xi_i$ implies that each $\xi_i$ is a 
continuous function $U' \to \R$
and so, applying \cite[Remark 9]{ParusinskiRainer15} and \cite[Theorem A.1]{Parusinski:2020aa} (see also \cite{ParusinskiRainerAC}),
we may conclude that each $\xi_i$ is of Sobolev class $\xi_i \in W^{1,p}(U')$, 
for all $1 \le p < \frac{n}{n-1}$,
and $\|\xi_i\|_{W^{1,p}(U')}$ depends uniformly on $f$; see \Cref{rem:uniformity}.
In view of \cite[Theorem 1.2]{MalySwansonZiemer03}, we obtain that
each graph $\Ga(\xi_i)$ is countably $\cH^{d-1}$-rectifiable and
\begin{equation} \label{eq:Hausdorff2}
   \cH^{d-1}(\{(x',\xi_i(x')) : x' \in E\}) = \int_E \sqrt{1 + |\nabla \xi_i(x')|^2} \,dx',
\end{equation}
for all measurable subsets $E \subseteq U'$.
The right-hand side of \eqref{eq:Hausdorff2} depends uniformly
on $f$ (see \Cref{rem:uniformity}).

Now it is easy to conclude

\begin{theorem} \label{prop:Sobolev}
  Let $(M^{|\infty},K,b)$ be admissible data and $\nn$ the associated integer.
  Let $f : K \to \R$ be any $M^{|\nn+1}$-smooth $\cC^\infty$-function such that $\|f\|_{K} \ge b$.
  There is a finite cover of $K$ by rectangular boxes $U$ with corresponding orthogonal coordinates $(x',x_d)$
  such that $Z_f\cap U$ is contained in the graphs $\{x_d = \xi_i(x')\}$ of at most $2\nn$
  continuous functions $\xi_i$
  of Sobolev class $W^{1,p}$,
  for all $1 \le p < \frac{2\nn}{2\nn-1}$,
  whose $W^{1,p}$-norm depends uniformly on $f$ (see \Cref{rem:uniformity}).
  Furthermore, \eqref{eq:Hausdorff2} holds.
\end{theorem}

Note that the Sobolev regularity in the results of \cite{ParusinskiRainer15,Parusinski:2020aa}
used above is optimal.

\begin{remark}
  It is evident from the above arguments
  that the statement of \Cref{prop:Sobolev} remains true for $\cC^\infty$-functions $f : K \to \R$
  that satisfy
  \[
       \|f\|_{j,K} \le \frac{1}{2 \de_K^j} \|f\|_{0,K}, \quad \text{ for some } j \ge 2,
  \]
  where $2\nn$ is replaced by $j-1$.
  Use \cite[Theorem 3(ii)]{Yomdin:1984dg}.
\end{remark}

\begin{remark} \label{rem:uniformity}
  In general,
  the coefficients $a_j$ of the polynomial \eqref{eq:preparation} will no longer be $M^{|\nn+1}$-smooth (cf.\ \cite{Acquistapace:2014wf}).
  But the choice of the functions $a_j$ (and $u$) can be made to depend linearly and continuously (with respect to the
  Whitney $\cC^\infty$ topology) on $f$ (cf.\ \cite{Mather:1968vt}).
  Furthermore, we have uniform bounds for the $W^{1,p}$-norm of the functions $\xi_i$ in terms of
  the $\cC^{2\nn}$-norm of the $a_j$ (cf.\ \cite{ParusinskiRainer15} and \cite{Parusinski:2020aa}).
\end{remark}

\subsection{Further remarks}

Let $(M^{|\infty},K,b)$ be admissible data and $\nn$ the associated integer.

\begin{remark} \label{rem:uniformity2}
    By definition, the number $\nn$ depends only on $(M^{|\infty},K,b)$.
    The above results on $Z_f$ remain unchanged as long as $f$ satisfies all the respective assumptions.
    E.g.,
    the bound on $\cH^{d-1}(Z_f)$ in \eqref{eq:Hausdorff} is the same as long as $f$
    fulfills the assumptions of \Cref{prop:Hausdorff}.

    Let $f : K \to \R$ be an $M^{|\nn+1}$-smooth function with $\|f\|_K\ge b$.
    If we set $V(f) := \sup_{x,y \in K} |f(x)-f(y)|$, then $V(f) \le 2 \|f\|_{K}$ and hence
    \[
      \|f-c\|_{K} \ge \frac{V(f-c)}{2} = \frac{V(f)}{2} \quad \text{ for all } c \in \R.
    \]
    On the other hand,
    \[
      \|f-c\|_{K} \le M_0 \quad \text{ if } |c| \le M_0 - \|f\|_{K}.
    \]
    So, if $2b \le V(f)$,
    we get the same uniform results for the level sets $Z_{f-c} = f^{-1}(c)$ for all $|c| \le M_0 - \|f\|_K$.
\end{remark}

\begin{remark}
  If we assume that
  $f : K \to \R$ is $(\de_K^{-1}\mu^{|\infty}, M_0)$-smooth with $\|f\|_K \ge b$
  and accordingly
  \begin{equation*}
        \Si_{\mu^{|\infty}}(j_0+1,\infty) > e, \qquad (j_0 = \lceil \log (\tfrac{b}{2M_0})^{-1}\rceil_{\N}),
  \end{equation*}
  instead of \eqref{eq:multdimcond},
  then
  we get all the above results with both
  the $\mu^{|\infty}$-degree
  $\fd_{\mu^{|\infty}}(\frac{b}{2M_0})$ and the constant $C$ in \eqref{eq:Hausdorff}
  independent of $\de_K$.
\end{remark}

\begin{remark} \label{rem:sectioning}
    A closer inspection of the proof of \Cref{prop:numberofzeros} reveals that 
    it would suffice to assume that each restriction $f|_\ell : K \cap \ell \to \R$ of $f$ is $M^{|\nn+1}$-smooth, 
    where $\ell$ is any affine line intersecting $K$. 
    But this is not far from the assumption that $f : K \to \R$ is $M^{|\nn+1}$-smooth:
    for simplicity assume that $f : K \to \R$ is of class $\cC^\infty$ (which is needed in \Cref{prop:Hausdorff} and 
    \Cref{prop:Sobolev} anyway). Then the uniformity of the bounds and the polarization inequality (\cite[7.13.1]{KM97}), 
    \[
        \sup_{v \in  \ol B_1} |d^j_v f(x)| \le \|d^jf(x)\|_{L^j(\R^d,\R)} \le (2e)^j \sup_{v \in  \ol B_1} |d^j_v f(x)|,
    \]
    where $d^j_vf(x):= \p_t^j|_{t=0} f(x+tv)$, 
    imply that  $f : K \to \R$ is $M^{|\nn+1}$-smooth, after slight modification of $M^{|\nn+1}$.
    (Note that, by results of Boman \cite{Boman67}, the assumption that the function is $\cC^\infty$ is actually not necessary, at least if 
        the domain of the function is open.)
    
    This is somewhat reminiscent of a result of Bochnak and Siciak \cite{Bochnak70,BochnakSiciak71,Siciak70}
    that a function is real analytic
    if its restrictions to all affine lines are real analytic. 
    For general (even quasianalytic) weights, this is however not true if the bounds are not uniform (i.e. they depend on the 
    affine lines); see \cite{Jaffe16} and \cite{Rainer:2019aa}.
\end{remark}

\section{Remez inequality for functions with controlled derivatives} \label{sec:Remez}

In this section, we prove a Remez-type inequality for $M^{|N}$-smooth (and $\mu^{|N}$-smooth) functions in several variables
and derive several consequences.
Our results are based on a univariate version due to \cite{NazarovSodinVolberg04} which we recall in slightly modified form
in \Cref{thm:NSV}.

\subsection{Definitions and conventions}

Let $\mu^{|\infty}= (\mu_j)_{j\ge 1}$ be an admissible weight.
Let us assume that there is an increasing continuous
function $\tilde \mu : [1,\infty) \to (0,\infty)$ that is (piecewise) $\cC^1$ such that
\[
  \mu_j = \tilde \mu(j), \quad j\ge 1.
\]
This is no real restriction, since we may always take $\tilde \mu$ piecewise affine and work consistently with the left derivative
at points, where $\tilde \mu$ is not differentiable.

Once we have $\tilde \mu$, we define (following \cite{NazarovSodinVolberg04})
\begin{equation*}
   \ga_{\tilde \mu}(n) := \sup_{1 \le s\le n} \frac{s  \tilde \mu'(s)}{\tilde \mu(s)} \quad \text{ and } \quad \Ga_{\tilde \mu}(n):= 4 e^{4+\ga_{\tilde \mu}(n)},
\end{equation*}
for all positive integers $n$.
Note that $\ga_{\tilde \mu}$ and $\Ga_{\tilde \mu}$ depend on the choice of $\tilde \mu$ which is not unique.

In this section, we will again make use of the terminology introduced in \Cref{sec:finitedeterminacy} (see \Cref{def:admissibledata}).
But here it is better to replace $b$ by $2 b$.
So let \emph{$(M^{|\infty},K,2b)$ be admissible data and $\nn$ the associated integer}.
Recall that this means the following:
$M^{|\infty}$ is a full admissible weight, $K \subseteq \R^d$ a convex body, and $b>0$ such that
\begin{equation} \label{eq:multdimcond2b}
   \Si_{\mu^{|\infty}}(j_0+1,\infty) > \de_K e, \qquad (j_0 = \lceil \log (\tfrac{b}{M_0})^{-1}\rceil_{\N}).
\end{equation}
The associated nonnegative integer $\nn$ is given by
\begin{equation*}
   \nn:= \fd_{\de_K \mu^{|\infty}}(\tfrac{b}{M_0}).
\end{equation*}
It satisfies \eqref{eq:fdet} and \eqref{eq:fdet2} with $b$ replaced by $2b$,
in particular,
\begin{equation} \label{eq:fdetR}
  \fd_{\de_K \mu^{|\infty}}(\tfrac{b}{M_0}) = \fd_{\de_K \mu^{|\nn+1}}(\tfrac{b}{M_0}) = \nn.
\end{equation}

In all occurrences of the Remez inequality for functions with controlled derivatives,
the constant $\Ga_{\tilde \mu}(2 \nn)^{2\nn}$ will appear.
If $\fd_{\de_K \mu^{|\infty}}(\frac{b}{M_0}) =\nn =0$ it is undefined.
For ease of notation, we set
\begin{equation}
   \CN:= \Ga_{\tilde \mu}(2 \nn)
\end{equation}
with the interpretation $\CN^{2\nn}=: \infty$ if $\nn=0$; see \Cref{rem:Remezinfinity}.

\subsection{Remez inequality for univariate functions}

We recall and slightly modify \cite[Theorem B]{NazarovSodinVolberg04}.
The statement in \cite{NazarovSodinVolberg04} involves the so-called Bang degree
$\fn_f$ (which depends on $f$)
and it is formulated for a quasianalytic full admissible weight $M^{|\infty}$ with $M_0=1$.
Here we work with $\fd_{\mu^{|\infty}}(\frac{b}{M_0}) = \fd_{\mu^{|\nn+1}}(\frac{b}{M_0})$ (which is independent of $f$) instead of $\fn_f$.

\begin{theorem} \label{thm:NSV}
  Let $(M^{|\infty},[0,1],2b)$ be admissible data and $\nn$ the associated integer.
  Let $f: [0,1] \to \R$ be any $M^{|2\nn}$-smooth function such that $\|f\|_{[0,1]}\ge b$.
  Then for any interval $I \subseteq [0,1]$ and any Lebesgue measurable set $E \subseteq I$ with $|E|>0$ we have
  \begin{equation} \label{eq:NSV}
     \|f\|_I \le \Big(\frac{\CN \, |I|}{|E|}\Big)^{2 \nn} \|f\|_E.
  \end{equation}
\end{theorem}

\begin{remark}\label{rem:Remezinfinity}
  If $\fd_{\mu^{|\infty}}(\frac{b}{M_0}) = \nn =0$, then, by the interpretation $\CN^{2\nn}= \infty$,
  \eqref{eq:NSV} is trivially true.
\end{remark}

\begin{proof}[Proof of \Cref{thm:NSV}]
  First of all, we may assume that $M_0= 1$, by dividing $f$ by $M_0$.
  In fact, $\bar f:=\frac{1}{M_0} f$ is $(\mu^{|2\nn},1)$-smooth and $\|\bar f\|_{[0,1]} \ge \frac{b}{M_0}=:\bar b>0$
  so that \eqref{eq:multdimcond2b}, \eqref{eq:NSV}, and the integer $\nn$ remain unchanged.

  By \Cref{rem:Remezinfinity}, we may assume that $\nn>0$ and
  follow the proof of \cite{NazarovSodinVolberg04};
  it involves only derivatives up to order $2 \fd_{\mu^{|\infty}}(b) =  2 \fd_{\mu^{|\nn+1}}(b) = 2\nn$ (cf.\
  \eqref{eq:fdetR} with $K=[0,1]$ and $M_0=1$).

  In the only place, where the definition of $\fn_f$ actually plays a role in the argument,
  we use the following estimate (cf.\ \cite[p.\ 72]{NazarovSodinVolberg04}):
  \begin{equation} \label{eq:NSV1}
     \min_{t \in [0,1]} b_0(t) = \min_{t \in [0,1]} \sup_{j \ge 0}\frac{|f^{(j)}(t)|}{e^j M_j}  > e^{- \nn-1}.
  \end{equation}
  Let us justify \eqref{eq:NSV1}. Suppose it is not true, i.e.,
  $\min_{t \in [0,1]} b_0(t) \le  e^{-\nn-1}$.
  On the other hand, $\max_{t \in [0,1]} b_0(t) \ge b \ge e^{-j_0}$, by the definition of $j_0$ (recall that $M_0=1$).
  We have $j_0 < \fd_{\mu^{|\nn+1}}(b) + 1=\nn+1$ (cf.\ \Cref{sec:degree}),
  that is $e^{-\nn-1}< e^{-j_0}$.
  Since $b_0$ is continuous (cf.\ \Cref{lem:Bang}),
  there is a monotonic sequence $x_j \in [0,1]$ such that $b_0(x_j) = e^{-j}$ for all
  $j_0 \le j \le \nn+1$ (cf.\ the proof of \Cref{thm:Bang}).
  By \Cref{lem:Bang}, we have
  \(
    |x_{j}-x_{j-1}|> \frac{1}{e \mu_j}
  \)
  for $j_0+1 \le j\le \nn+1$
  so that
  \[
  1 \ge \sum_{j=j_0 +1}^{\nn+1} |x_{j}-x_{j-1}| > \frac{1}{e} \sum_{j=j_0 +1}^{\nn+1}\frac{1}{\mu_j},
  \]
  and thus $\nn = \fd_{\mu^{|\infty}}(\frac{b}{M_0}) \ge \nn+1$, a contradiction.
  Hence \eqref{eq:NSV1} is proved.
\end{proof}

\begin{remark} \label{eq:olGa}
  As pointed out in \cite{NazarovSodinVolberg04},
  if $\ol \ga_{\tilde \mu}:=\sup_{s\ge 1} \frac{s\tilde \mu'(s)}{\tilde \mu(s)} < \infty$
  (as is the case for $\tilde \mu(s) =s$ or $\tilde \mu(s) = s (\log(s+e))^\de$ if $0<\de\le 1$),
  then in \eqref{eq:NSV} the constant $\CN$ can be replaced by
  $\ol \Ga_{\tilde \mu} = 4 e^{4+\ol \ga_{\tilde \mu}}$, provided that $\nn \ne 0$ (see \Cref{rem:Remezinfinity}).
\end{remark}

\subsection{Remez inequality for multivariate functions}

Next we prove a Remez-type inequality for multivariate functions with controlled derivatives.
The proof is inspired by the technique of \cite{Brudnyi:1973un}.

\begin{theorem} \label{thm:Remez}
  Let $(M^{|\infty},K,2b)$ be admissible data and $\nn$ the associated integer.
  Let $f : K \to \R$ be any $M^{|2\nn}$-smooth function and
  $L \subseteq K$ any convex body
  such that $\|f\|_{L} \ge b$.
  If $E \subseteq L$ is a Lebesgue measurable subset with $|E|>0$, then
  \begin{equation*}
     \|f\|_{L} \le \Big(
     \frac{\CN\, |L|^{1/d}}{|L|^{1/d} - (|L|-|E|)^{1/d}}
     \Big)^{2 \nn} \|f\|_E.
  \end{equation*}
\end{theorem}

\begin{proof}
  Let $\cF(E,L) = \cF_{M^{|2\nn}}(E,L,b)$ denote the set of all $M^{|2\nn}$-smooth functions $f : K \to \R$
  satisfying $\|f\|_{L}  \ge b$ and $\|f\|_E\le 1$.
  For $\la>0$ consider
  \[
    \cR(\la,L) = \cR_{M^{|2\nn}}(\la,L,b) := \sup_{\substack{E \subseteq L \\ |E|\ge \la}} \sup_{f \in \cF(E,L)} \|f\|_{L}.
  \]
  We claim that
  \begin{equation} \label{eq:BG}
    \cR(\la,L) \le
     \Big(
     \frac{\CN\, |L|^{1/d}}{|L|^{1/d} - (|L|-\la)^{1/d}}
     \Big)^{2 \nn}.
  \end{equation}
  Fix a measurable subset $E \subseteq L$ with $|E|\ge \la$ and $f \in \cF(E,L)$.
  There is
  $x_0 \in L$ with $\|f\|_{L} = |f(x_0)|$.
  Let $\ell$ be any half-line emanating from $x_0$
  and let $x_1$ be the (other) intersection point of $\ell$ with $\p L$.
  Define $g : [0,1] \to \R$ by $g(t):= f(x_0 + t(x_1-x_0))$.
  As seen in the proof of \Cref{prop:numberofzeros}, $g : [0,1] \to \R$ is $(\de_K \mu^{|2\nn},M_0)$-smooth
  and $\|g\|_{[0,1]} \ge |g(0)| = |f(x_0)| \ge b$.
  Thus (in view of \eqref{eq:multdimcond2b} and since $\|f\|_E \le 1$) \Cref{thm:NSV} implies
  \[
   \|f\|_{L} = |f(x_0)|=|g(0)| \le \Big(\frac{\CN\, \cL^1(L \cap \ell)}{\cL^1(E \cap \ell)}\Big)^{2 \nn},
  \]
  because of the scaling properties of the $1$-dimensional Lebesgue measure $\cL^1$ (induced on $\ell$);
  note that $\ga_{\de_K \tilde \mu}=\ga_{\tilde \mu}$ and hence $\Ga_{\de_K \tilde \mu}=\Ga_{\tilde \mu}$.
  Taking the essential infimum over all half-lines emanating from $x_0$,
  the supremum over all $f \in \cF(E,L)$, and the
  supremum over all measurable $E \subseteq L$ with $|E|\ge \la$, we find
  \[
   \cR(\la,L) \le
    \Big(\CN \sup_{\substack{E \subseteq L \\ |E|\ge \la}}
    \essinf_{\ell} \frac{\cL^1(L \cap \ell)}{\cL^1(E \cap \ell)}\Big)^{2 \nn},
  \]
  and, using \cite[Lemma 3 and Remark 2]{Brudnyi:1973un}, we conclude \eqref{eq:BG}.

  Now we may prove the statement of the theorem.
  Let $f$, $L$, and $E$ be as in the assumptions of the theorem.
  Since $|E|>0$, we have $\|f\|_{E} > 0$, by \Cref{cor:zeromeasure}.
  Then $F:= \frac{f}{\|f\|_E}$ is $\frac{1}{\|f\|_E} M^{|2\nn}$-smooth and satisfies
  $\|F\|_{L} \ge \frac{b}{\|f\|_E}$ and $\|F\|_E=1$.
  Thus $F \in \cF_{\frac{1}{\|f\|_E}M^{|2\nn}}(E,L,\frac{b}{\|f\|_E})$; note that \eqref{eq:multdimcond2b}, $\nn$, and $\CN$
  remain unchanged.
  Then \eqref{eq:BG} implies
  \begin{equation*}
     \|F\|_{L} \le \Big(
     \frac{\CN\, |L|^{1/d}}{|L|^{1/d} - (|L|-|E|)^{1/d}}
     \Big)^{2\nn}.
  \end{equation*}
  Since $\|F\|_{L} = \frac{\|f\|_L}{\|f\|_E}$, this completes the proof.
\end{proof}

\subsection{An important consequence} \label{sec:homogeneous}

The next result is a simple but important consequence of \Cref{thm:Remez}.

For its formulation, it is convenient to adapt our terminology.
We say that the pair \emph{$(\mu^{|\infty},K)$ is admissible data with associated integer $\nn$} if
$\mu^{|\infty} = (\mu_j)_{j\ge 1}$ is an admissible weight and $K \subseteq \R^d$ a convex body such that
\begin{equation} \label{eq:multdimcond1}
   \Si_{\mu^{|\infty}}(1,\infty) > \de_K e
\end{equation}
and
\begin{equation*}
   \nn := \fd_{\de_K \mu^{|\infty}}(1).
\end{equation*}
As before, $\CN:= \Ga_{\tilde \mu}(2 \nn)$ and  $\CN^{2\nn}=: \infty$ if $\nn=0$.

We will specialize \Cref{thm:Remez} to the case $b = M_0 = \|f\|_K$, that is to
$(\mu^{|2\nn},\|f\|_K)$-smooth functions $f : K \to \R$. 
Recall that, by the definition in \Cref{ssec:Msmooth}, 
$f : K \to \R$ is called 
$(\mu^{|2\nn},\|f\|_K)$-smooth if $\|f\|_K >0$ and
\begin{equation*}
    \|f\|_{j,K} \le \|f\|_K \cdot \mu_1 \mu_2 \cdots \mu_{j}, \quad 1 \le j \le 2\nn.
\end{equation*}

\begin{theorem} \label{cor:Remez2}
  Let $(\mu^{|\infty},K)$ be admissible data and $\nn$ the associated integer.
  Let $f : K \to \R$ be any $(\mu^{|2\nn},\|f\|_K)$-smooth function.
  If $E \subseteq K$ is a Lebesgue measurable subset with $|E|>0$, then
  \begin{equation} \label{eq:Remezinequality1}
     \|f\|_{K} \le \Big(
     \frac{\CN\, |K|^{1/d}}{|K|^{1/d} - (|K|-|E|)^{1/d}}
     \Big)^{2 \nn} \|f\|_E.
  \end{equation}
\end{theorem}

\begin{proof}
  Set $L = K$ and $b=M_0 := \|f\|_{K}$ in \Cref{thm:Remez}.
\end{proof}

There is no \emph{a priori} condition on $\|f\|_K$, except $\|f\|_{K}>0$, in \Cref{cor:Remez2}.
Visibly, \eqref{eq:Remezinequality1} is invariant under
the action of $\R^*$ on $f$.

It can happen that $\fd_{\de_K \mu^{|\infty}}(1) = \nn= 0$, in which case \Cref{cor:Remez2} contains no information
(cf.\ \Cref{rem:Remezinfinity}).
In fact, this occurs precisely if $\de_K \mu_1 \le 1/e$.

\subsection{Volume of sublevel sets}

\Cref{cor:Remez2} has several interesting corollaries.
We begin with a bound on the growth of the volume of sublevel sets.

\begin{corollary} \label{cor:Remez3}
Let $(\mu^{|\infty},K)$ be admissible data and $\nn$ the associated integer.
Let $f : K \to \R$ be any $(\mu^{|2\nn},\|f\|_K)$-smooth function.
Then the sublevel set $S_t := \{x \in K : |f(x)|\le t\}$ satisfies
\begin{equation} \label{eq:sublevel}
     |S_t| \le \CN d\, |K|\,
     \Big(\frac{t}{\|f\|_{K}}\Big)^{\frac{1}{2 \nn}}, \quad t>0,
  \end{equation}
with the understanding that the right-hand side is $\infty$ if $\nn=0$.
\end{corollary}

\begin{proof}
    We may assume that $\nn \ne 0$ and that $|S_t| >0$ (the inequality being trivial otherwise).
   Apply \Cref{cor:Remez2} to $E = S_t$.
   Then, putting $\th:= \frac{|S_t|}{|K|}$,
    \[
    \|f\|_{K} \le \Big(
    \frac{\CN}{1 - (1-\th)^{1/d}}
    \Big)^{2 \nn} t,
    \]
    and consequently,
    \[
    \frac{\th}{d}  \le 1 - (1-\th)^{1/d} \le
    \CN\,
    \Big(\frac{t}{\|f\|_{K}}\Big)^{\frac{1}{2 \nn}}.
    \]
  The statement follows.
\end{proof}

\Cref{cor:Remez3} implies useful estimates for the distribution function and the decreasing rearrangement of $f$
(more generally, of $f|_E$ for a measurable subset $E \subseteq K$). 
Recall that the distribution function of $f: K \to \R$ is defined by
\[
    d_f(t):= |\{x \in K : |f(x)|>t\}| = |K|-|S_t|
\]
and the decreasing rearrangement of $f$ by 
\[
    f^*(y) := \inf\{t>0 : d_f(t) \le y\}.
\]

\begin{corollary} \label{cor:distribution}
Let $(\mu^{|\infty},K)$ be admissible data and $\nn$ the associated integer.
Let $f : K \to \R$ be any $(\mu^{|2 \nn},\|f\|_K)$-smooth function.
Let $E\subseteq K$ be a Lebesgue measurable subset with $|E|>0$.
Then 
\begin{equation}
    \label{eq:Remez4keyE}
       (f|_E)^*(|E|\, \la) \ge \|f\|_K \Big(\frac{|E|}{|K|} \cdot \frac{1-\la}{\CN d} \Big)^{2\nn}, \quad \la \in (0,1),
\end{equation}
where the right-hand side is identically zero if $\nn=0$.
\end{corollary}

\begin{proof}
    Assume that $\nn\ne 0$.
    We have 
    \[
        d_{f|_E}(t) := |\{x \in E : |f(x)|> t\}| = |E| - |S_t \cap E|= |E|\, \la_t,
    \]
    where $\la_t := 1- |S_t \cap E|/|E|$.
    Let $s_\la$ denote the right-hand side of \eqref{eq:Remez4keyE}.
    By \Cref{cor:Remez3},
    \begin{equation}\label{eq:distribution}
    d_{f|_E}(t)=|E|\, \la_t > |E|\, \la, \quad\text{ if } t \in (0,s_\la).
    \end{equation}
    Indeed, by \eqref{eq:sublevel},
   \begin{align*}
       |E|(1- \la_t) =  |S_t \cap E| \le |S_t| &\le \CN d\, |K|\,
      \Big(\frac{t}{\|f\|_{K}}\Big)^{\frac{1}{2 \nn}}
      \\
                                               &< \CN d\, |K|\,
      \Big(\frac{s_\la}{\|f\|_{K}}\Big)^{\frac{1}{2 \nn}}
      =|E|(1- \la).
   \end{align*} 
    Now \eqref{eq:distribution} implies $(f|_E)^*(|E|\, \la) \ge s_\la$,
    and \eqref{eq:Remez4keyE} is proved.
\end{proof}

\subsection{Comparison of $L^p$-norms}

Let us write
\begin{align*}
  \|f\|_{L^p(E)}^\sharp &:= \Big( \frac{1}{|E|} \int_E |f(x)|^p \, dx \Big)^{1/p}, \quad 0 < p <\infty,
  \\
  \|f\|_{L^\infty(E)}^\sharp &:= \esssup_{E}|f|,
\end{align*}
for the normalized $L^p$-norms (respectively, quasinorms if $0<p<1$) of $f$ on a measurable set $E$ with $0<|E|<\infty$.
Then, as a consequence of H\"older's inequality,
\[
  \|f\|_{L^q(E)}^\sharp \le \|f\|_{L^p(E)}^\sharp, \quad \text{ if } 0< q \le p \le \infty.
\]
Indeed, since $p/q\ge 1$, 
\begin{align*}
    \int_E |f|^q \, dx \le |E|^{1-q/p} \Big( \int_E |f|^{p}\, dx\Big)^{q/p}. 
\end{align*}

For functions with controlled derivatives also suitable opposite inequalities hold.

\begin{corollary}
   \label{cor:Remez4}
  Let $(\mu^{|\infty},K)$ be admissible data and $\nn$ the associated integer.
  Let $f : K \to \R$ be any $(\mu^{|2\nn},\|f\|_K)$-smooth function.
Let $E \subseteq K$ be a Lebesgue measurable subset with $|E|>0$.
  Then, for all $0< q< p \le \infty$,
  \begin{equation} \label{eq:LpLqmaster}
        \|f\|_{L^p(K)}^\sharp \le  \Big(\frac{\CN d\, |K|}{|E|}\Big)^{2 \nn(1-\frac{q}{p})}
        (2q\, \nn+1)^{\frac{1}{q}-\frac{1}{p}}  
        (\|f\|_{L^q(K)}^\sharp)^{\frac{q}p} (\|f\|_{L^q(E)}^\sharp)^{1-\frac{q}{p}}.
    \end{equation}
  In particular, 
  \begin{equation} \label{eq:LpLqK}
         \|f\|_{L^p(K)}^\sharp \le  (\CN d)^{2 \nn(1-\frac{q}{p})}
        (2q\, \nn+1)^{\frac{1}{q}-\frac{1}{p}}  
        \|f\|_{L^q(K)}^\sharp
  \end{equation}
  and
 \begin{equation} \label{eq:LpLqE}
       \|f\|_{L^p(K)}^\sharp \le  \Big(\frac{ \CN d\, |K|}{|E|}\Big)^{2\nn} \,
       (2q \nn+1)^{\frac{1}{q}}\,  \|f\|_{L^q(E)}^\sharp.
    \end{equation}
\end{corollary}

\begin{proof}
    We may assume that $\nn\ne 0$; otherwise the statements are trivially true.
    Note that both \eqref{eq:LpLqK} and \eqref{eq:LpLqE} follow from \eqref{eq:LpLqmaster}:
    it suffices to specialize to $E=K$ 
    and to use $\|f\|_{L^q(K)}^\sharp \le \|f\|_{L^p(K)}^\sharp$, respectively.

    Let us prove \eqref{eq:LpLqmaster}.
    We begin with the case $p=\infty$. 
    Using \eqref{eq:Remez4keyE}, we get
    \begin{align*}
        \frac{1}{|E|} \int_{E} |f(x)|^q \, dx &= \frac{1}{|E|} \int_0^{|E|} \big((f|_E)^*(y)\big)^q \, dy
       =  \int_0^1 \big((f|_E)^*(|E|\, \la)\big)^q \, d\la
       \\
                               &\ge    \Big(\frac{|E|}{\CN d\, |K|}\Big)^{2q \nn}\int_0^{1} (1-\la)^{2q \nn} \, d\la \,  \|f\|_K^q
       \\
                               &= \Big(\frac{|E|}{\CN d\, |K|}\Big)^{2q \nn} \frac{1}{2q \nn+1}\, \|f\|_{K}^q,
    \end{align*}
    whence 
    \begin{equation} \label{eq:LinftyLqmaster}
       \|f\|_K \le  \Big(\frac{\CN d\, |K|}{|E|}\Big)^{2 \nn}
        (2q\, \nn+1)^{\frac{1}{q}}  
        \|f\|_{L^q(E)}^\sharp
    \end{equation}
    which is \eqref{eq:LpLqmaster} for $p=\infty$.

    If $0< q < p < \infty$, then \eqref{eq:LinftyLqmaster} implies
     \begin{align*}
     \MoveEqLeft 
     \frac{1}{|K|} \int_{K} |f(x)|^p \,dx
        \le  \|f\|_{K}^{p-q} \frac{1}{|K|}\int_{K} |f(x)|^q \,dx = \|f\|_{K}^{p-q}\, (\|f\|_{L^q(K)}^\sharp)^q
        \\
        &\le \Big(\Big(\frac{\CN d\, |K|}{|E|}\Big)^{2 \nn}
        (2q\, \nn+1)^{\frac{1}{q}}\Big)^{p-q} \, (\|f\|_{L^q(K)}^\sharp)^q (\|f\|_{L^q(E)}^\sharp)^{p-q}
     \end{align*}
     from which \eqref{eq:LpLqmaster} follows easily.
     The proof is complete.
\end{proof}

\begin{remark}
  Bourgain proved
  in \cite{Bourgain:1991te} the following inequality for polynomials:
  {\it
  Let $K \subseteq \R^d$ be a convex body of volume $1$ and $p \in \R[x_1,\ldots,x_d]$ a polynomial of degree $n$.
  Then, for each $q>0$,
  \begin{equation*}
     \|p\|_{L^q(K)} \le C(n,q)\, \|p\|_{L^1(K)}.
  \end{equation*}
  More precisely,
  \begin{equation*}
     \|p\|_{L^\psi(K)} \le C_1\, \|p\|_{L^1(K)},
  \end{equation*}
  where $L^\psi$ it the Orlicz space with Orlicz function $\psi(t) = \exp(t^{C_2/n})-1$,
  where the constants $C_1$ and $C_2$ are absolute.
  }

  These inequalities have been generalized by Brudnyi \cite{Brudnyi_1999} to analytic functions,
  where the role of $n$ is played by the \emph{analytic degree}; see \Cref{sec:analyticdegree}.

  In contrast to \eqref{eq:LpLqmaster}, \eqref{eq:LpLqK}, and \eqref{eq:LpLqE},
  here the constants do not depend on the dimension $d$.
  Following Bourgain's approach,
  it should be possible to obtain versions of \eqref{eq:LpLqmaster}, \eqref{eq:LpLqK}, 
  and \eqref{eq:LpLqE} that are independent of the dimension:
  use \Cref{cor:Remez2} for $d=1$ and $K=[0,1]$ to get a replacement for \cite[Lemma 3.1]{Bourgain:1991te}.
\end{remark}

\subsection{A bound for the mean oscillation of $\log |f|$}

Let $K \subseteq \R^d$ be a convex body.
Recall that the \emph{mean oscillation} of a locally integrable function $g : K \to \R$ over a ball $B \subseteq K$ is
\[
  \on{mo}_B(g) := \frac{1}{|B|} \int_B |g(x)-g_B|\, dx,
\]
where
\[
  g_B := \frac{1}{|B|} \int_B g(x)\, dx.
\]
We have the following bound for the mean oscillation of $\log |f|$ if
$f$ has suitably controlled derivatives.

\begin{corollary} \label{cor:Remez6}
  Let $(\mu^{|\infty},K)$ be admissible data and $\nn$ the associated integer.
  Let $f : K \to \R$ be any $(\mu^{|2\nn},\|f\|_K)$-smooth function.
  Then, for each ball $B \subseteq K$,
  \begin{equation} \label{eq:mo}
     \on{mo}_B(\log |f|) \le 4 \nn  \Big(\log \Big(\frac{\CN d\, |K|}{|B|}\Big) +1 \Big)
  \end{equation}
  where the right-hand side is interpreted as $\infty$ if $\nn = 0$.
\end{corollary}

\begin{proof}
Observe that
\[
  \on{mo}_B(\log |f|) \le  \frac{2}{|B|} \int_B \Big|\log \frac{|f(x)|}{\|f\|_K}\Big| \, dx.
\]
Indeed,
\begin{align*}
   \Big| \log |f| - \frac{1}{|B|} \int_B \log |f|\, dx \Big| &\le \Big| \log |f| - \log \|f\|_K \Big| + \Big| \log \|f\|_K  - \frac{1}{|B|} \int_B \log |f|\, dx \Big|
   \\
   &= \Big| \log \frac{|f|}{\|f\|_K} \Big| + \Big| \frac{1}{|B|} \int_B \log \|f\|_K  -  \log |f|\, dx \Big|
   \\
   &\le \Big| \log \frac{|f|}{\|f\|_K} \Big| + \frac{1}{|B|} \int_B \Big| \log \frac{|f|}{\|f\|_K} \Big| \, dx.
\end{align*}
  We may assume that $\nn\ne 0$.
Now, by \eqref{eq:Remez4keyE},
\begin{align*}
   \frac{1}{|B|}\int_B \Big|\log \frac{|f(x)|}{\|f\|_K}\Big| \, dx
   &= \frac{1}{|B|} \int_0^{|B|} \Big|\log \frac{(f|_B)^*(y)}{\|f\|_K}\Big| \, dy
    \\
    &=  \int_0^{1} \Big|\log \frac{(f|_B)^*(|B|\, \la)}{\|f\|_K}\Big| \, d\la
    \\
    &\le  2\nn  \int_0^{1} -\log \Big(\frac{|B|}{|K|}\cdot  \frac{1-\la}{\CN d} \Big)    \, d\la
    \\
    &= 2\nn  \Big(\log \Big(\frac{\CN d\, |K|}{|B|}\Big) +1 \Big),
\end{align*}
and \eqref{eq:mo} follows.
\end{proof}

\Cref{cor:Remez6} has similarity with the log-BMO property of analytic functions 
\cite[Corollary 1.10]{Brudnyi_1999}.
But
from \eqref{eq:mo} it seems not possible to deduce that $\log |f|$ has bounded mean oscillation,
since the right-hand side tends to infinity if the radius of $B$ tends to zero.
Also using \eqref{eq:mo} in the case $K = B$ does not help, because then $\nn$ becomes eventually zero
if $B$ gets small enough.

This is related to the fact that we do not have good enough control 
away from points in $K$, where $\|f\|_K$ is attained.
Indeed, if $x \in K$ is such that $|f(x)| = \|f\|_K$, then
\[
    \sup_{\substack{\text{balls }B \subseteq K \\ x \in B}}  \on{mo}_B(\log |f|)  \le 4 \nn\, \big(\log (\CN d)+1\big)
\]
as follows from the next corollary.

\begin{corollary} \label{cor:BMO}
    Let $(\mu^{|\infty},K)$ be admissible data and $\nn$ the associated integer.
  Let $f : K \to \R$ be any $(\mu^{|2\nn},\|f\|_K)$-smooth function.
  Then, for each convex body $L \subseteq K$ such that $\|f\|_L = \|f\|_K$,
  \begin{equation} \label{eq:moL}
     \on{mo}_L(\log |f|) \le 4 \nn\,  \big(\log(\CN d) +1\big)
  \end{equation}
  where the right-hand side is interpreted as $\infty$ if $\nn = 0$.
\end{corollary}

\begin{proof}
   If $L \subseteq K$ is a convex body such that $\|f\|_L = \|f\|_K$ 
   and $E$ is a measurable subset of $L$ with $|E|>0$, 
   then we get 
   \[
   \|f\|_{L} \le \Big(
     \frac{\CN\, |L|^{1/d}}{|L|^{1/d} - (|L|-|E|)^{1/d}}
     \Big)^{2 \nn} \|f\|_E.
   \]
   from \Cref{thm:Remez} (in analogy to \Cref{cor:Remez2}). 
For $E := |\{x \in L : |f(x)| \le t\}|$ we may conclude
\[
    |E| \le 
    \CN d\, |L|\,
     \Big(\frac{t}{\|f\|_{L}}\Big)^{\frac{1}{2 \nn}}, \quad t>0,
\]
(in analogy to \Cref{cor:Remez3})
and thus 
\[
 (f|_L)^*(|L|\, \la) \ge \|f\|_L \Big(\frac{1-\la}{\CN d}\Big)^{2\nn}, \quad \la \in (0,1),    
\]
(in analogy to \Cref{cor:distribution}).  Using this estimate in the computations of the proof of \Cref{cor:Remez6},
yields \eqref{eq:moL}.
\end{proof}

\section{Complementary results}

In this section, we compare the $\mu^{|\infty}$-degree to the polynomial and the analytic degree, respectively.
Moreover, we discuss the existence of non-quasianalytic bump functions and
show how the technique from \Cref{sec:univariate} can be used to extract information on critical points.

\subsection{The $\mu^{|\infty}$-degree vs.\ the polynomial degree} \label{sec:polynomialdegree}

 What is the relation between the $\mu^{|\infty}$-degree  and the usual degree of a polynomial?

 Since most of our results are based on restriction to one-dimensional sections,
 let us assume that $d=1$.
  Let $p \in \R[t]$ be a univariate polynomial of degree $n$.
  We want to find upper and lower bounds in $n$ for the $2 \mu^{|\infty}$-degree $\fd_{2 \mu^{|\infty}}(1)$
  for any suitable full admissible weight $(\mu^{|\infty},M_0)$ such that $p : [-1,1]\to \R$ 
  is $(\mu^{|\infty},M_0)$-smooth.
  (It is natural to take $N=\infty$, since all derivatives of high order of $p$ are identically zero anyway.
  The factor $2$ appears, since $\de_{[-1,1]}=2$.)

  By Markov's inequality (\cite{Markov:1889aa}, \cite{MarkovGrossmann1916}),
  \begin{equation} \label{eq:Markov}
    \|p^{(k)}\|_{[-1,1]} \le \frac{n^2 (n^2 -1^2)(n^2 -2^2)  \cdots (n^2-(k-1)^2)}{1 \cdot 3 \cdot 5  \cdots  (2k-1)} \|p\|_{[-1,1]},
    \quad 1 \le k \le n.
  \end{equation}
  Thus a most natural full admissible weight $(\mu^{|\infty},M_0)$ such that $p : [-1,1] \to \R$ is $(\mu^{|\infty},M_0)$-smooth is
  \begin{equation} \label{eq:vspoly}
     M_0 := \|p\|_{[-1,1]}, \quad \mu_j := n^2, \quad j \ge 1.
  \end{equation}
  In fact, $(\mu_j)_j$ has to be increasing so that any admissible choice must satisfy $\mu_j \ge n^2$ for all $j \ge 1$. 
  Making $(\mu_j)_j$ bigger increases also $\fd_{2\mu^{|\infty}}(1)$ and in that way we
  could make $\fd_{2\mu^{|\infty}}(1)$ as large as we please.
  It is also natural to take the second parameter equal to $1$, since $\frac{b}{M_0}=1$ for the choice $b=\|p\|_{[-1,1]}$
  and $b \mapsto \fd_{2 \mu^{|\infty}}(b)$ is decreasing.

  For the choice \eqref{eq:vspoly} we have
  \[
    \Si_{\mu^{|\infty}}(1,m) = \sum_{j=1}^m \frac{1}{\mu_j} =  \frac{m}{n^2}
  \]
  and, consequently,
  \[
    \fd_{\mu^{|\infty}}(1) = \lfloor e n^2 \rfloor.
  \]
  Note that in this case 
  \begin{equation} \label{eq:2fd}    
      2 \fd_{\mu^{|\infty}}(1) \le \fd_{2\mu^{|\infty}}(1)   \le 2 \fd_{\mu^{|\infty}}(1)+1;
  \end{equation}    
  the first inequality is always true, because $(\mu_j)_j$ is increasing.

  We get a better result for complex polynomials. Let $p \in \C[z]$ be a univariate polynomial of degree $n$.
  By Bernstein's inequality (\cite{Bernstein:1926aa}, \cite{Riesz:1914aa}),
  \begin{equation*}
      \|p^{(k)}\|_{\ol D_1} \le \frac{n!}{(n-k)!} \|p\|_{\ol D_1}, \quad k \ge 1,
  \end{equation*}
  where $\ol D_1 := \{z \in \C : |z|\le 1\}$ is the closed unit disk.
  Thus $p|_{\ol D_1}$ is $(\mu^{|\infty},M_0)$-smooth for 
  \[
      M_0 := \|p\|_{\ol D_1}, \quad \mu_j = C n, \quad j\ge 1,
  \]
  where the constant $C>0$ accounts for the conversion of the bounds for complex derivatives to real partial derivatives.
  In this case, we find 
\[
    \fd_{\mu^{|\infty}}(1) = \lfloor C e n\rfloor,
\]
and \eqref{eq:2fd} is still valid.

\begin{remark}
    Strictly speaking a $(\mu^{|\infty}, M_0)$-smooth function is real valued by definition (see \Cref{ssec:Msmooth}).
    Its definition clearly makes sense for complex valued functions as well, 
    but crucial results of the paper are based on Rolle's theorem.
\end{remark}

\subsection{The $\mu^{|\infty}$-degree vs.\ the analytic degree} \label{sec:analyticdegree}

The \emph{analytic degree} $d_f(2\ep)$ of a holomorphic function $f$ on the disk $D_{1+2\ep}:= \{z \in \C : |z|< 1+2\ep\}$, for some $\ep> 0$,
is defined in \cite{Brudnyi_1999} as the best constant $d$ in the inequality
\[
  \|f\|_I \le \Big(\frac{4\, |I|}{|E|}\Big)^d \|f\|_E,
\]
where $I$ is any interval in the intersection of a real affine line in $\C$ with the unit disk $D_1$ and $E \subseteq I$ is any
measurable subset with $|E|>0$.

By the Cauchy estimates,
\[
  \|f^{(k)}\|_{\ol D_{1}} \le \frac{k!\, \|f\|_{D_{1+\ep}}}{\ep^{k+1}}, \quad k \ge 1.
\]
So $f$ on $\ol D_1$ is $(\mu^{|\infty},M_0)$-smooth for
\[
  M_0 := \frac{\|f\|_{D_{1+\ep}}}{\ep}, \quad \mu_j :=  \frac{C j}{\ep}, \quad j \ge 1,
\]
where the constant $C>0$ accounts for the conversion of the bounds for complex derivatives to real partial derivatives.
For $\tilde \mu(s)=s$ we have $\ol \ga_{\tilde \mu}=1$ and $\ol \Ga_{\tilde \mu}=4 e^5$; see \Cref{eq:olGa}.
By \Cref{thm:NSV}, we find (as in the proof of \Cref{thm:Remez})
\[
  \|f\|_I \le \Big(\frac{4e^5 |I|}{|E|}\Big)^{2\fd_{2\mu^{|\infty}}(\ep)} \|f\|_E.
\]
Since $(4 e^5)^{2\fd_{2\mu^{|\infty}}(\ep)} \le 4^{10\, \fd_{2\mu^{|\infty}}(\ep)}$, we conclude that
\begin{equation*}
   d_f(2\ep) \le 10\, \fd_{2\mu^{|\infty}}(\ep).
\end{equation*}
We do not know if there is a lower bound for $d_f(2\ep)$ in terms of $\fd_{2\mu^{|\infty}}(\ep)$.

\begin{remark}
  For $\mu^{|\infty} =(j)_{j\ge 1}$,
    $\Si_{\mu^{|\infty}}(1,n)$ is the partial harmonic series $\sum_{j=1}^n \frac{1}{j} =: H_n$,
    and, more generally, $\Si_{\mu^{|\infty}}(m,n) = H_n - H_{m-1}$ if $1 < m \le n$.
    Let $x \ge 1$ and define the positive integer $n(x)$ by
    \begin{equation} \label{eq:Comtet}
      H_{n(x)} \le x < H_{n(x)+1}.
    \end{equation}
    Comtet \cite{Comtet67} (see also Boas and Wrench \cite{BoasWrench71}) showed that, for $x\ge 2$,
    \begin{equation*}
      \Big\lfloor e^{x-\ga} - \frac{1}{2} - \frac{3}{2} \frac{1}{e^{x-1}-1} \Big\rfloor
       \le n(x) \le \Big\lfloor e^{x-\ga} - \frac{1}{2} + \frac{1}{12} \frac{1}{e^{x-1}-1} \Big\rfloor,
    \end{equation*}
    where $\ga$ is the Euler--Mascheroni constant,
    which determines one or two possible values of $n(x)$ for any $x\ge 2$.
    With this formula it is not difficult to find explicit estimates of $\fd_{a \mu^{|\infty}}(b)$, for $a,b>0$.
    Note that the possible equality $H_{n(x)} = x$ in \eqref{eq:Comtet}, in contrast to the strict inequality in the definition of $\fd_{a \mu^{|\infty}}(1)$,
    might effect a deviation of at most $1$ between $n(ae)$ and $\fd_{a \mu^{|\infty}}(1)$.
\end{remark}

\subsection{Conditions for non-quasianalytic bump functions}

Let $M^{|\infty}$ be a non-quasianalytic full admissible weight.
We may infer from \Cref{prop:numberofzeros}
a quantitative necessary condition for
the existence of $M^{|\infty}$-smooth functions $f$ with compact support contained in the interior $K^\o$ of the convex body $K$.
Of course, this is most informative if $K$ has equal width in all directions, e.g., if $K$ is a ball.

\begin{corollary} \label{cor:cpsupp}
  Let $M^{|\infty}$ be a non-quasianalytic full admissible weight, $K \subseteq \R^d$ a convex body, and $b>0$.
  Suppose that $f : K \to \R$ is an $M^{|\infty}$-smooth function compactly supported in $K^\o$ and $\|f\|_K\ge b$.
  Then
  \begin{equation*}
     \Si_{\mu^{|\infty}} (j_0+1,\infty) \le \de_K e, \qquad (j_0 = \lceil \log (\tfrac{b}{2M_0})^{-1} \rceil_{\N}).
  \end{equation*}
\end{corollary}

\begin{proof}
   \Cref{prop:numberofzeros} shows that \eqref{eq:multdimcond} must be violated.
\end{proof}

\begin{remark} \label{rem:complement}
  For completeness, we sketch a proof of the following fact:
  \emph{Let $Z$ be any non-empty closed subset of $\R^d$ and $M^{|\infty}=(M_j)_{j\ge 0}$ a non-quasianalytic full admissible weight.
  There exist $f \in \cC^\infty(\R^d)$ and $A>0$
  such that $\|f^{(\al)}\|_{\R^d}  \le A^{|\al|+1} M_{|\al|}$, for all  $\al \in \N^d$, and
  $Z_f=Z \subseteq Z(f^{(\al)})$ for all $\al \ne 0$.}

  By \cite[Lemma 2.3 and Corollary 3.5]{Rainer:2021aa},
  there exists a non-quasianalytic full admissible weight $L^{|\infty}=(L_k)_{k \ge 0}$  such that
  \begin{equation} \label{eq:LM}
    \Big(\frac{M_j}{L_j}\Big)^{1/j} \to \infty.
  \end{equation}
  By \cite[Theorem 1.4.2]{Hoermander83I} or \cite[Proposition 3.11]{Rainer:2021aa},
  there is a $\cC^\infty$-function $0\le \vh \le 1$ with $\vh(0)=1$ and support contained in the unit ball $B_1$ and $A>0$ such that
  $\|\vh^{(\al)}\|_{\R^d}  \le A^{|\al|+1} L_{|\al|}$ for all $\al$.
  For every $x\in \R^d \setminus Z$, let $d(x) := \frac{1}{2} \inf_{z \in Z} |x-z|$ and
  $\vh_x(y):= \vh(\frac{y-x}{d(x)})$. Then
  $\vh_x(x)=1$, $U_x := \{y : \vh_x(y)\ne 0\}\subseteq \R^d \setminus Z$, and
  \[
    \|\vh_x^{(\al)}\|_{\R^d} = \frac{1}{d(x)^{|\al|}} \|\vh^{(\al)}\|_{\R^d} \le \frac{A^{|\al|+1}}{d(x)^{|\al|}} L_{|\al|}.
  \]
  The family $\{U_x\}$ forms an
  open cover of $\R^d \setminus Z$
  which admits a countable subcover $\{U_n:=U_{x_n}\}$.
  Let $\vh_n := \vh_{x_n}$, $d_n := 1/d(x_n)$, and
  choose constants $s_n>0$ such that
  \[
    d_n^j s_n \le \frac{M_j}{L_j} \quad \text{ for all } n,j \ge 1.
  \]
  This is possible by \eqref{eq:LM}, since for each $d_n$ we find $j_n$ such that $d_n \le \big(\frac{M_j}{L_j}\big)^{1/j}$ for all $j > j_n$;
  so we may take
  \[
    s_n := \min \Big\{\min_{j \le j_n} \frac{M_j}{d_n^j L_j},1\Big\}.
  \]
  Then $f := \sum_{n\ge 1} \frac{s_n}{2^n} \vh_n$ converges uniformly in all derivatives,
  \begin{align*}
     \Big|\sum_{n\ge 1} \frac{s_n}{2^n} \vh_n^{(\al)}(x)\Big|
     &\le
     \sum_{n\ge 1} \frac{s_n}{2^n}\cdot A^{|\al|+1} d_n^{|\al|} L_{|\al|}
     \le
     A^{|\al|+1} M_{|\al|},
  \end{align*}
  so that $f$ defines a $\cC^\infty$-function on $\R^d$ with $f^{(\al)} = \sum_{n\ge 1} \frac{s_n}{2^n} \vh_n^{(\al)}$
  and $\|f^{(\al)}\|_{\R^d}  \le A^{|\al|+1} M_{|\al|}$ for all $\al$.
  Since the $U_n$ cover $\R^d \setminus Z$,
  $f$ is strictly positive on $\R^d \setminus Z$ and
  vanishes, together with all partial derivatives $f^{(\al)}$, on $Z$.
\end{remark}

\subsection{Critical points}

Let $M^{|\infty}$ be a full admissible weight and
$f : K \to \R$ an $M^{|\infty}$-smooth function, where $K \subseteq \R^d$ is a convex body.
In order to get quantitative information on the critical points of $f$,
one can consider the function $g:=|\nabla f|^2 = \sum_{j=1}^d (\p_j f)^2$
and apply the results of \Cref{sec:multidim} to $g$.
Note that $g$ is $\hat M^{|\infty}$-smooth for a full admissible weight
$\hat M^{|\infty}$ which can be computed from $M^{|\infty}$
in view of the Fa\`a di Bruno formula.

But the one-dimensional analysis from \Cref{sec:univariate}
allows us to extract information in a more direct way.

\begin{proposition} \label{prop:critical}
  Let $M^{|\infty}= (M_j)_{j\ge 0}$ be a full admissible weight, $K \subseteq \R^d$ a convex body, and $b>0$
  such that
  \begin{equation} \label{eq:multdimcondcrit}
     \Si_{\mu^{|\infty}}(j_0+1,\infty) > 2\de_K e, \qquad (j_0 = \lceil \log (\tfrac{b}{2M_0})^{-1}\rceil_{\N}).
  \end{equation}
  Then $\nn:= \fd_{2\de_K \mu^{|\infty}}(\tfrac{b}{2M_0})$ is a nonnegative integer.
  Let $f : K \to \R$ be any $M^{|\nn+1}$-smooth function such that $\|f\|_{K} \ge b$.
  Let $\ell$ be any affine line that meets $B$ (i.e.\ the ball contained in $\{x \in K : |f(x)|>\frac{b}{2}\}$ from \Cref{prop:numberofzeros})
  such that either $\ell \cap Z_f \ne \emptyset$ or
  $V(f|_\ell)\ge 2b$.
  Then $f$ has at most $2\nn$ critical points on $\ell$.
  Thus at all but possibly $2\nn$ points of $\ell$ the level sets of $f$ are $\cC^{\nn+1}$-submanifolds.
\end{proposition}

Notice the factor $2$ in \eqref{eq:multdimcondcrit} and in the definition of $\nn$.

\begin{proof}
   Let $\ell$ be an affine line that meets $B$.
   If no zero of $f$ lies on $\ell$, we may assume without loss of generality that
   $f$ is positive on $\ell$ and that
   $c:= \min_{\ell \cap K} f > 0$.
   The assumption
   $V(f|_\ell) \ge 2b$ guarantees that we can
   replace $f$ by $f-c$; clearly the set of critical points of $f$ and $f-c$ is the same.
   Indeed, cf.\ \Cref{rem:uniformity2},
   \[
   b\le \frac{V(f|_\ell)}{2} \le \|f-c\|_{\ell \cap K} = \max_{\ell \cap K} f-c \le M_0 -c \le M_0,
   \]
   and $\|f-c\|_{j,K} = \|f\|_{j,K}$ for all $j\ge 1$.
   Since $f-c$ has a zero on $\ell$, we may assume
   from now on that $f$ has a zero on $\ell$.

   Now we restrict $f$ to $\ell$ and work with the $(\de_K \mu^{|\nn+1},M_0)$-function $g : [0,1]\to \R$
   as in the proof of
   \Cref{prop:numberofzeros}.
   Since $\ell$ meets $B$, we have $\|g\|_{[0,1]}> \frac{b}{2}$.
   And $g$ has at least one zero $s_0$ in $[0,1]$, because $f$ vanishes on $\ell$.
   As in \Cref{prop:numberofzeros},
   we see that the number of critical points of $g$ left and right of $s_0$ is finite and bounded by $\nn$.
   Since each critical point of $f$ on $\ell$ is a critical point of $g$, the proof is complete.

   Let us explain in more detail why we need the factor $2$ in the definition of $\nn$.
   To this end, assume that $t_1 \le t_2 \le \cdots \le t_{m_r}$ is an increasing enumeration of the critical points of $g$ that lie to the right of $s_0$.
   By Rolle's theorem, we find $t_1 = s_1 \le s_2 \le \cdots \le s_{m_r-1}$ such that $g^{(j)}(s_j)=0$ for $0\le j \le m_r-1$.
   Choose $s_{-1} \in [0,1]$ such that $|g(s_{-1})| > \frac{b}{2}$.
   If $s_{-1}< s_0$, we find, as in the proof of \Cref{cor:Bang1},
   \[
     1 \ge \sum_{j= 0}^{m_r-1} |s_j - s_{j-1}| > \frac{1}{e} \Si_{\de_K \mu^{|\infty}}(j_0+1,m_r),
   \]
   whence $m_r \le \fd_{\de_K \mu^{|\infty}}(\frac{b}{2M_0})$.
   If $s_{-1}> s_0$, then
   \[
    \sum_{j= 0}^{m_r-1} |s_j - s_{j-1}| = (s_{-1}-s_0) + \sum_{j=1}^{m_r-1} (s_j - s_{j-1}) \le 2
   \]
   so that we may conclude that $m_r \le \fd_{2 \de_K \mu^{|\infty}}(\frac{b}{2M_0})$.
   In any case $m_r \le \nn$.
\end{proof}

\begin{remark}
  If $M$ has additional regularity properties so that the implicit function theorem holds in the
  associated Denjoy--Carleman class, then the submanifolds are of the respective Denjoy--Carleman class; see e.g.\
  \cite{RainerSchindl14}.
\end{remark}

A comprehensive quantitative study of the critical and near-critical values of differentiable mappings
can be found in \cite{Yomdin:1983wg}.

\subsection*{Acknowledgements}
Part of the work on this paper has been done at the Fields Institute in Toronto, Canada, during the
\emph{Thematic Program on Tame Geometry, Transseries and Applications to Analysis and Geometry (January 1 -- June 30, 2022)}.
I am grateful for the kind hospitality and the excellent working conditions.
Also I thank A. Debrouwere and B. Prangoski for pointing out a mistake in an earlier version of the paper.
The author was supported by the Austrian Science Fund (FWF), Project P 32905-N.


\begin{thebibliography}{10}

\bibitem{Acquistapace:2014wf}
F.~Acquistapace, F.~Broglia, M.~Bronshtein, A.~Nicoara, and N.~Zobin,
  \emph{{Failure of the Weierstrass Preparation Theorem in quasi-analytic
  Denjoy{\textendash}Carleman rings}}, Advances in Mathematics \textbf{258}
  (2014), 397--413.

\bibitem{Bang53}
T.~Bang, \emph{The theory of metric spaces applied to infinitely differentiable
  functions}, Math. Scand. \textbf{1} (1953), 137--152.

\bibitem{Bernstein:1926aa}
S.~N. Bernstein, \emph{{Lecons sur les propri\'et\'es extr\'emales et la
  meilleure approximation des fonctions analytiques d'une variable re\'elle}},
  Collection Borel, Paris, 1926.

\bibitem{BoasWrench71}
R.~P. Boas, Jr. and J.~W. Wrench, Jr., \emph{Partial sums of the harmonic
  series}, Amer. Math. Monthly \textbf{78} (1971), 864--870.

\bibitem{Bochnak70}
J.~Bochnak, \emph{Analytic functions in {B}anach spaces}, Studia Math.
  \textbf{35} (1970), 273--292.

\bibitem{BochnakSiciak71}
J.~Bochnak and J.~Siciak, \emph{Analytic functions in topological vector
  spaces}, Studia Math. \textbf{39} (1971), 77--112.

\bibitem{Boman67}
J.~Boman, \emph{Differentiability of a function and of its compositions 
with functions of one variable}, Math. Scand. \textbf{20} (1967), 249--268.

\bibitem{Bourgain:1991te}
J.~Bourgain, \emph{On the distribution of polynomials on high dimensional
  convex sets}, Geometric aspects of functional analysis. Proceedings of the
  Israel seminar (GAFA) 1989-90, Berlin etc.: Springer-Verlag, 1991, Lecture
  Notes in Math. 1469, pp.~127--137 (English).

\bibitem{Brudnyi:1999tb}
A.~Brudnyi, \emph{Local inequalities for plurisubharmonic functions}, Ann.
  Math. (2) \textbf{149} (1999), no.~2, 511--533 (English).

\bibitem{Brudnyi_1999}
\bysame, \emph{On local behavior of analytic functions}, Journal of Functional
  Analysis \textbf{169} (1999), no.~2, 481--493.

\bibitem{Brudnyi:1973un}
Yu.~A. Brudnyi and M.~I. Ganzburg, \emph{A certain extremal problem for
  polynomials in {$n$} variables}, Izv. Akad. Nauk SSSR Ser. Mat. \textbf{37}
  (1973), 344--355.

\bibitem{Comtet67}
L.~Comtet, \emph{Problems and {S}olutions: {S}olutions of {A}dvanced
  {P}roblems: 5346}, Amer. Math. Monthly \textbf{74} (1967), no.~2, 209.

\bibitem{Hoermander83I}
L.~H\"ormander, \emph{The analysis of linear partial differential operators.
  {I}}, Grundlehren der Mathematischen Wissenschaften [Fundamental Principles
  of Mathematical Sciences], vol. 256, Springer-Verlag, Berlin, 1983,
  Distribution theory and Fourier analysis.

\bibitem{Jaffe16}
E.~Y. Jaffe, \emph{Pathological phenomena in {D}enjoy-{C}arleman classes}, 
Canad. J. Math. \textbf{68} (2016), no.~1, 88--108. 

\bibitem{KM97}
A.~Kriegl and P.~W. Michor, \emph{The convenient setting of global analysis}, 
Mathematical Surveys and Monographs, vol.~53, American Mathematical Society, 
Providence, RI, 1997, \url{http://www.ams.org/online\_bks/surv53/}.

\bibitem{MalySwansonZiemer03}
J.~Mal{{\'y}}, D.~Swanson, and W.~P. Ziemer, \emph{The co-area formula for
  {S}obolev mappings}, Trans. Amer. Math. Soc. \textbf{355} (2003), no.~2,
  477--492 (electronic).

\bibitem{Markov:1889aa}
A.~Markov, \emph{{On a problem of D. I. Mendeleev}}, Zapiski Imp. Akad. Nauk.
  \textbf{62} (1889), 1--24, (Russian).

\bibitem{MarkovGrossmann1916}
V.~Markov and J.~Grossmann, \emph{\"{U}ber {P}olynome, die in einem gegebenen
  {I}ntervalle m\"{o}glichst wenig von {N}ull abweichen}, Math. Ann.
  \textbf{77} (1916), no.~2, 213--258.

\bibitem{Mather:1968vt}
J.~N. Mather, \emph{Stability of {$C^{\infty }$} mappings. {I}. {T}he division
  theorem}, Ann. of Math. (2) \textbf{87} (1968), 89--104.

\bibitem{NazarovSodinVolberg04}
F.~Nazarov, M.~Sodin, and A.~Volberg, \emph{Lower bounds for quasianalytic
  functions. {I}. {H}ow to control smooth functions}, Math. Scand. \textbf{95}
  (2004), no.~1, 59--79.

\bibitem{ParusinskiRainerAC}
A.~Parusi{{\'n}}ski and A.~Rainer, \emph{Regularity of roots of polynomials},
  Ann.\ Sc.\ Norm.\ Super.\ Pisa Cl.\ Sci.\ (5) \textbf{16} (2016), 481--517.

\bibitem{ParusinskiRainer15}
\bysame, \emph{Optimal {S}obolev regularity of roots of polynomials}, Ann.\
  Sci.\ \'Ec.\ Norm.\ Sup\'er.\ (4) \textbf{51} (2018), no.~5, 1343--1387,
  https://doi.org/10.24033/asens.2376.

\bibitem{Parusinski:2020aa}
\bysame, \emph{Selections of bounded variation for
  roots of smooth polynomials}, Sel. Math. New Ser. \textbf{26} (2020), no.~13,
  https://doi.org/10.1007/s00029-020-0538-z.

\bibitem{Rainer:2019aa}
A.~Rainer, \emph{Quasianalytic ultradifferentiability cannot be tested in lower dimensions}, 
Bull. Belg. Math. Soc. Simon Stevin \textbf{26} (2019), 505--517.

\bibitem{Rainer:2021aa}
\bysame, \emph{Ultradifferentiable extension theorems: a survey},
Expositiones Mathematicae \textbf{40} (2022), no.~3, 679--757,
  https://doi.org/10.1016/j.exmath.2021.12.001.

\bibitem{RainerSchindl14}
A.~Rainer and G.~Schindl, \emph{Equivalence of stability properties for
  ultradifferentiable function classes}, Rev. R. Acad. Cienc. Exactas Fis. Nat.
  Ser. A Math. RACSAM. \textbf{110} (2016), no.~1, 17--32.

\bibitem{Remes:1936wl}
E.~Remez, \emph{Sur une propri{\'e}t{\'e} extremale des polyn{\^o}mes de
  {Tchebychef}}, Commun. {Inst}. {Sci}. {Math}. et {Mecan}., {Univ}. {Kharkow}
  et {Soc}. {Math}. {Kharkow}, {IV}. {Ser}. 13, 93-95 (1936).

\bibitem{Riesz:1914aa}
M.~Riesz, \emph{{Eine trigonometrische Interpolationsformel und einige
  Ungleichungen f\"ur Polynome}}, Jahresberichte DMV \textbf{23} (1914),
  354--368.

\bibitem{Siciak70}
J.~Siciak, \emph{A characterization of analytic functions of {$n$} real
  variables}, Studia Math. \textbf{35} (1970), 293--297.


\bibitem{Yomdin:1983wg}
Y.~Yomdin, \emph{The geometry of critical and near-critical values of
  differentiable mappings}, Math. Ann. \textbf{264} (1983), 495--515 (English).

\bibitem{Yomdin:1984dg}
\bysame, \emph{{The zero set of an ``almost polynomial'' function}}, {Proc. Am.
  Math. Soc.} \textbf{90} (1984), no.~4, 538--542.

\bibitem{Yomdin:2014aa}
\bysame, \emph{Remez-type inequality for smooth functions}, Springer
  Proceedings in Mathematics {\&} Statistics, Springer International
  Publishing, 2014, pp.~235--243.

\end{thebibliography}

\def\cprime{$'$}
\providecommand{\bysame}{\leavevmode\hbox to3em{\hrulefill}\thinspace}
\providecommand{\MR}{\relax\ifhmode\unskip\space\fi MR }
\providecommand{\MRhref}[2]{%
  \href{http://www.ams.org/mathscinet-getitem?mr=#1}{#2}
}
\providecommand{\href}[2]{#2}

\end{document}